\documentclass[english]{smfart}
\usepackage{amsmath, amssymb, amsthm, amscd,smfenum,xspace,mathrsfs,url,upref,verbatim}
\usepackage[utf8]{inputenc} 
\usepackage[T1]{fontenc}
\usepackage{dsfont}
\usepackage{pstricks}
\usepackage{graphicx}
\usepackage{mathrsfs}
\usepackage{ tipa }
\usepackage{mathabx}
\usepackage{stmaryrd}
\usepackage[all,cmtip]{xy}
\usepackage{xfrac} 

\usepackage{ textcomp }

\usepackage{tikz-cd}
\usepackage{tikz}

\DeclareMathAlphabet{\mathpzc}{OT1}{pzc}{m}{it}

\let\hat=\widehat
\let\tilde=\widetilde
\numberwithin{equation}{subsection}
 
\newtheorem{theorem}{Theorem} 
 
\newtheorem*{question}{Question} 
\newtheorem{proposition}[equation]{Proposition}
\newtheorem{lemme}[equation]{Lemma}
\newtheorem{corollaire}[equation]{Corollary}

\theoremstyle{remarque}
\newtheorem{remarque}[equation]{Remark}

\theoremstyle{fact}
\newtheorem{fact}[equation]{Fact}
\newtheorem{definition}[equation]{Definition}

\DeclareMathOperator{\res}{\mathbf{res}} 
\DeclareMathOperator{\ob}{ob} 
\DeclareMathOperator{\rk}{rk} 
 
\DeclareMathOperator{\Gr}{Gr}
\DeclareMathOperator{\DR}{DR}

\DeclareMathOperator{\Hom}{Hom}

\DeclareMathOperator{\ord}{ord}
\DeclareMathOperator{\Max}{Max}
\DeclareMathOperator{\id}{id}
\DeclareMathOperator{\Ker}{Ker}

\DeclareMathOperator{\Spec}{Spec}

\DeclareMathOperator{\reduit}{red}

\def\cartesien{\ar@{}[rd]|{\square}}

\DeclareMathOperator{\RHom}{RHom}
\DeclareMathOperator{\Ext}{Ext}
\DeclareMathOperator{\spe}{sp}

\DeclareMathOperator{\GL}{GL}
\DeclareMathOperator{\Aut}{Aut}
\DeclareMathOperator{\Gal}{Gal}

\DeclareMathOperator{\mon}{mon}
\DeclareMathOperator{\Frac}{Frac}

\DeclareMathOperator{\End}{End}

\DeclareMathOperator{\univ}{univ}
\DeclareMathOperator{\Sing}{Sing}

\DeclareMathOperator{\Sl}{Sl}
\DeclareMathOperator{\Open}{Open}
\DeclareMathOperator{\Irr}{Irr}
\DeclareMathOperator{\Sol}{Sol}

\DeclareMathOperator{\an}{an}

\DeclareMathOperator{\nb}{nb}

\DeclareMathOperator{\LR}{LR}
\DeclareMathOperator{\Isom}{Isom}

\DeclareMathOperator{\alg}{alg}

\DeclareMathOperator{\Set}{Set}

\DeclareMathOperator{\Id}{Id}

\DeclareMathOperator{\Real}{Re}

\DeclareMathOperator{\St}{St}
\DeclareMathOperator{\irr}{irr}

\DeclareMathOperator{\Good}{Good}

\DeclareMathOperator{\colim}{colim}

\DeclareMathOperator{\Lie}{Lie}

\DeclareMathOperator{\iso}{iso}
\DeclareMathOperator{\fppf}{fppf}

\DeclareMathOperator{\rank}{rank}
\usepackage{enumerate}

\author[J.-B.~Teyssier]{Jean-Baptiste Teyssier}
\curraddr{Institut de Mathématiques de Jussieu. 4 place Jussieu. Paris.}
\email{jean-baptiste.teyssier@imj-prg.fr}
\title{Moduli of Stokes torsors and singularities of differential equations}

\usepackage[english]{babel}
\begin{document}

\begin{abstract}
Let  $\mathcal{M}$ be a meromorphic connection  with poles along a smooth divisor $D$ in a smooth algebraic variety. Let $\Sol \mathcal{M}$ be the solution complex of  $\mathcal{M}$. We prove that the good formal decomposition locus of $\mathcal{M}$ coincides with the locus where the restrictions to $D$ of $\Sol \mathcal{M}$ and $\Sol \End\mathcal{M}$ are local systems. By contrast to the very different natures of these loci (the first one is defined via algebra, the second one is defined via analysis), the proof of their coincidence is geometric. It relies on the moduli of Stokes torsors. 
\end{abstract}
\maketitle

The main problematic of this paper is to understand how the geometry of the Stokes phenomenon  in any dimension sheds light on the interplay between the singularities of a differential equation and the singularities of its solutions. \\ \indent
Consider an algebraic linear system $\mathcal{M}$ of differential equations with $n$ variables
$$
                \frac{\partial X}{\partial x_i }=\Omega_{i} X   \text{ \hspace{0.5cm} } i=1, \dots , n
$$
 where $\Omega_i $ is a  square matrix of size $r$ 
with coefficients into the ring $\mathds{C}[x_1, \dots , x_n][x^{-1}_n]$ of Laurent polynomials with poles along the hyperplane $D$ in $\mathds{C}^n$ given by $x_n=0$. 
At a point away from $D$, the holomorphic solutions of the system $\mathcal{M}$ are fully understood by means of Cauchy's theorem. At a point of $D$, the situation is much more complicated. It is still the source of  challenging unsolved problems. We call $D$ \textit{the singular locus of  $\mathcal{M}$}. Two distinguished open subsets of $D$ where the singularities of $\mathcal{M}$ are mild can be defined. \\ \indent
First, the set $\Good(\mathcal{M})$ of \textit{good formal decomposition points of $\mathcal{M}$} is the subset of $D$ consisting of points $P$ at the formal neighbourhood of which $\mathcal{M}$ admits a \textit{good decomposition}. For $P$ being the origin, and modulo ramification issues that will be neglected in this introduction, this means roughly that there exists a base change with coefficients in $\mathds{C}\llbracket x_1, \dots , x_n\rrbracket [x^{-1}_n]$ splitting $\mathcal{M}$ as a direct sum of well-understood systems easier to work with. \\ \indent
Good formal decomposition can always be achieved in the one variable case \cite{SVDP}. It is desirable in general because it provides a concrete description of the system, at least formally at a point. In the higher variable case however, it was observed in \cite{Sabbahdim} that $\mathcal{M}$ may not have  good formal decomposition at every point of $D$. Thus, the set $\Good(\mathcal{M})$ is  a non trivial invariant of $\mathcal{M}$. As proved by André \cite{Andre}, the set $\Good(\mathcal{M})$  is the complement in $D$ of a Zariski closed subset $F$ of $D$ either purely of codimension $1$ in $D$ or empty.
Traditionally, $F$ is called the \textit{Turning locus of $\mathcal{M}$}, by reference to the way the Stokes directions of $\mathcal{M}$ move along a small circle in $D$ going around a turning point. In a sense, \textbf{the good formal decomposition locus of $\mathcal{M}$ is the open subset of $D$ where the singularities of the system $\mathcal{M}$ are as simple as possible}. \\ \indent
To define the second distinguished subset of $D$ associated to $\mathcal{M}$, let us view $\mathcal{M}$ as a $\mathcal{D}$-module, that is a module over the Weyl algebra of differential operators. Let us denote by $\Sol \mathcal{M}$ the solution complex of the analytification of $\mathcal{M}$. Concretely, $\mathcal{H}^0\Sol \mathcal{M}$ encodes the holomorphic solutions of our differential system while the higher cohomologies of $\Sol \mathcal{M}$ keep track of higher Ext groups in the category of  $\mathcal{D}$-modules. As proved by Kashiwara \cite{Ka2}, the complex $\Sol \mathcal{M}$  is perverse. From  a theorem of Mebkhout 
\cite{Mehbgro}, the restriction of  $\Sol \mathcal{M}$  to $D$, that is, the \textit{irregularity complex of $\mathcal{M}$ along $D$}, denoted by $\Irr_{D}^{\ast} \mathcal{M}$ in this paper, is also perverse. In particular, $(\Sol \mathcal{M})_{|D}$ is a local system on $D$ away from a closed analytic subset of $D$. The \textit{smooth locus of $(\Sol \mathcal{M})_{|D}$} denotes the biggest open in $D$ on which $(\Sol \mathcal{M})_{|D}$ is a local system.
 In a sense, \textbf{the smooth locus of $(\Sol \mathcal{M})_{|D}$ is the open subset of $D$ where the singularities of the (derived) solutions of $\mathcal{M}$ are as simple as possible}. \\ \indent
As observed in \cite{phdteyssier}, the open set $\Good(\mathcal{M})$ is included in the smooth locus of  $(\Sol \mathcal{M})_{|D}$ and $(\Sol \End\mathcal{M})_{|D}$. The reverse inclusion was conjectured in \cite[15.0.5]{phdteyssier}. Coincidence of $\Good(\mathcal{M})$ with the smooth locus of $(\Sol \mathcal{M})_{|D}$ and $(\Sol \End\mathcal{M})_{|D}$ seems surprising at first sight, since goodness is an \textit{algebraic} notion whereas $\Sol \mathcal{M}$ is \textit{transcendental}. The main goal of this paper is to prove via \textit{geometric} means the following
\begin{theorem}\label{mainth}
Let $X$ be a smooth complex algebraic variety. Let $D$ be a smooth divisor in $X$. Let  $\mathcal{M}$ be a meromorphic connection on $X$ with poles along $D$. Then, the good formal decomposition locus of $\mathcal{M}$ is  the  locus of $D$ where $(\Sol \mathcal{M})_{|D}$ and $(\Sol \End\mathcal{M})_{|D}$ are local systems.
\end{theorem}
Other criteria detecting good points of meromorphic connections are available in the literature. Let us mention André's criterion \cite[3.4.1]{Andre} in terms of specialisations of Newton polygons. Let us also mention Kedlaya's criterion \cite[4.4.2]{Kedlaya1} in terms of the
variation of spectral norms under varying Gauss norms on the ring of formal power series. This criterion is numerical in nature. By contrast, the new criterion given by Theorem \ref{mainth} is \textit{transcendental}. Its sheaf theoretic flavour makes it possible to track the existence of turning points in the cohomology of the irregularity complex. For an application of this observation, let us refer to Theorem \ref{boundedness} below.\\ \indent
The main tool at stake in the proof of Theorem \ref{mainth} is geometric, via \textit{moduli of Stokes torsors} \cite{TeySke}. For a detailed explanation of the line of thoughts that brought them into the picture, let us refer to \ref{why}. 
Before stating an application of Theorem \ref{mainth} (see Theorem \ref{boundedness} below), we explain how these moduli are used by giving the main ingredients of the proof of Theorem 1 in dimension $2$. In that case, we have to show the goodness of a point $0\in D$ given that $(\Sol \mathcal{M})_{|D}$ and $(\Sol \End\mathcal{M})_{|D}$ are local systems in a neighbourhood of $0$.
The main problem is to extend the good formal decomposition of $\mathcal{M}$  across $0$. This decomposition can be seen as a system of linear differential equations $\mathcal{N}$ defined in a neighbourhood of a small disc $\Delta^{\ast}$ of $D$ punctured at $0$. \\ \indent
To show that $\mathcal{N}$  extends across $0$, we first construct via Stokes torsors a moduli space  $\mathcal{X}$ parametrizing very roughly systems defined in a neighbourhood of $\Delta$ and formally isomorphic to $\mathcal{M}$ along $\Delta$. A distinguished point of $\mathcal{X}$ is given by $\mathcal{M}$ itself.
Similarly, we construct a moduli space  $\mathcal{Y}$ parametrizing roughly systems defined in a neighbourhood of $\Delta^{\ast}$ and formally isomorphic to $\mathcal{M}_{|\Delta^{\ast}}$ along $\Delta^{\ast}$. Two distinguished points of $\mathcal{Y}$  are $\mathcal{M}_{|\Delta^{\ast}}$ and $\mathcal{N}$. Restriction from $\Delta$ to $\Delta^{\ast}$ provides a morphism of algebraic varieties $\res : \mathcal{X} \longrightarrow \mathcal{Y}$. The problem of extending $\mathcal{N}$ is then the problem of proving that $\res$ hits $\mathcal{N}$. The moduli $\mathcal{X}$ and $\mathcal{Y}$ have the wonderful property that the tangent  map $T_{\mathcal{M}} \res$ of $\res$ at $\mathcal{M}$ is exactly  the map 
$$
\Gamma(\Delta, \mathcal{H}^1\Sol \End\mathcal{M}) \longrightarrow \Gamma(\Delta^{\ast}, \mathcal{H}^1\Sol \End\mathcal{M} )
$$
 associating  to $s\in \Gamma(\Delta, \mathcal{H}^1\Sol \End\mathcal{M})$ the restriction of $s$ to $\Delta^{\ast}$. In this geometric picture, the smoothness of $(\mathcal{H}^1\Sol \End\mathcal{M})_{|D}$ around $0$ thus translates into the fact that 
\textit{$T_{\mathcal{M}} \res$ is an isomorphism of vector spaces}. Since $\mathcal{X}$ and $\mathcal{Y}$ are smooth, we deduce that  $\res$  is \textit{étale} at the point $\mathcal{M}$. Thus, the image of $\res$ in $\mathcal{Y}$ contains a non empty open set. We prove  furthermore (see Theorem \ref{cestuneimmersionfermee} below) that $\res$ is a closed immersion, so its image is closed in $\mathcal{Y}$. Since $\mathcal{Y}$  is irreducible, we conclude that $\res$ is surjective, which proves the existence of the  sought-after extension of $\mathcal{N}$. \\ \indent

Let us now describe an application of Theorem \ref{mainth}. Let $X$ be a smooth variety over a finite field of characteristic $p>0$. Let $\ell\neq p$  be a prime number. As proved by Deligne  \cite{HenKerz}, there is only a finite number of semi-simple $\ell$-adic local systems on $X$ with prescribed rank, bounded ramification at infinity and up to a twist by a character coming from the base field. A natural question is to look for a differential analogue of this finiteness result. Let $X$ be a smooth complex proper algebraic variety. Let 
$\mathcal{M}$ be a  meromorphic connection on $X$. In this situation, H. Esnault and A. Langer asked whether it is possible to control the resolution of turning points of $\mathcal{M}$ by means of $X$, the rank of $\mathcal{M}$ and the irregularity of $\mathcal{M}$. In dimension $2$, this question amounts to bound the number of blow-ups needed to eliminate the turning points of $\mathcal{M}$. To the author's knowledge, this question is still widely open. If such a bound exists in dimension $2$, the number of turning points of $\mathcal{M}$ should in particular be bounded by a quantity depending only on the surface $X$, the rank of $\mathcal{M}$ and the irregularity of $\mathcal{M}$. As an application of Theorem \ref{mainth}, we give such a bound in a relative situation, thus providing the first evidence for a positive answer to H. Esnault and A. Langer's question. This is the following 
\begin{theorem}\label{boundedness}
Let $S$ be a smooth complex algebraic curve. Let $0\in S$. Let $p: \mathcal{C}\longrightarrow S$ be a relative smooth proper curve of genus $g$. Let $\mathcal{M}$ be a meromorphic connection on  $\mathcal{C}$ with poles along the fibre $\mathcal{C}_0$ of $p$ above $0$. Let $r_D(\mathcal{M})$ be the highest generic slope of $\mathcal{M}$ along $\mathcal{C}_0$. Then, the number of turning points of $\mathcal{M}$ along $\mathcal{C}_0$ is bounded by $ 8(\rank \mathcal{M})^2(g+1)r_D(\mathcal{M} )$.
\end{theorem}
To prove Theorem \ref{boundedness}, the main tools are Theorem  \ref{mainth} and a new boundedness result for nearby slopes \cite{NearbySlope} suggested by the $\ell$-adic picture  \cite{HuTeyssier}. See remark \ref{lienladic} for details.

A crucial step in the proof of Theorem \ref{mainth} is to understand the geometry of the restriction map for Stokes torsors. This is achieved by Theorem \ref{cestuneimmersionfermee} below. To state it, let $X$ be a smooth complex algebraic variety. Let $D$ be a normal crossing divisor in $X$. Let $\mathcal{M}$ be a good meromorphic connection on $X$ with poles along $D$. Let $p_D : \tilde{X}\longrightarrow X$ be the fibre product of the real blow-ups of $X$ along the components of $D$. For every subset $A\subset D$, put $\partial A := p^{-1}_D(A)$. Let $\St_{\mathcal{M}}^{<D}$ be the Stokes sheaf  of $\mathcal{M}$ (see section \ref{defmoduli} for details). This is a sheaf of complex unipotent algebraic groups on $\partial D$. Then, we have the following
\begin{theorem}\label{cestuneimmersionfermee}
Let $U\subset V\subset D$ be non empty open subsets in $D$ such that $V$ is connected. Then, the natural morphism
$$
\xymatrix{
H^{1}( \partial V, \St_{\mathcal{M}}^{<D})            \ar[r] &  
H^{1}(\partial U, \St_{\mathcal{M}}^{<D})   
}
$$
is a closed immersion of affine schemes of finite type over $\mathds{C}$.
\end{theorem}
Let us finally give an application of Theorem \ref{cestuneimmersionfermee} to degenerations of irregular singularities. Let $X$ be a smooth algebraic variety and let $D$ be a germ of smooth divisor at $0\in X$. Let $\mathcal{M}$ be a germ of meromorphic connection defined in a neighbourhood of $D$ in $X$ and with poles along $D$.  Motivated by Dubrovin's conjecture and the study of Frobenius manifolds, Cotti, Dubrovin and Guzzetti \cite{CDG} studied how much information on the Stokes data of $\mathcal{M}$  can be retrieved from the restriction of $\mathcal{M}$  to a smooth curve $C$ transverse to $D$ and passing through $0$. \\ \indent
Under the assumption that $\mathcal{M}_{\hat{D}}$ splits as a direct sum of regular connections twisted by meromorphic functions $a_1, \dots , a_n\in \mathcal{O}_X(\ast D)$ with simple poles along $D$, they proved  that \textit{the Stokes data of the restriction $\mathcal{M}_{|C}$ determine in a bijective way the Stokes data of $\mathcal{M}$ in a small  neighbourhood of $0$ in $D$}. This is striking, since the numerators of the $a_i - a_j$ may vanish at $0$, thus inducing a discontinuity at $0$ in the configuration of the Stokes directions. Using different methods, this was reproved by Sabbah in \cite[Th 1.4]{SabbahsurCDG}. In this paper, we give a short conceptual proof of a stronger version of Cotti, Dubrovin and Guzzetti's injectivity theorem: we don't make any assumption on the shape of $\mathcal{M}_{\hat{D}}$, nor do we suppose that $D$ is smooth, nor do we assume that $C$ is transverse to $D$. The price to pay for this generality is the use of resolution of turning points, as proved in the fundamental work of 
Kedlaya \cite{Kedlaya2} and  Mochizuki \cite{Mochizuki1}. The intuition that the techniques developed in this paper could be applied to the questions considered by Cotti, Dubrovin and Guzzetti is due to C. Sabbah. \\ \indent
To state our result, let us recall that a \textit{$\mathcal{M}$-marked connection} is the data of a pair $(M, \iso)$ where $M$ is a germ of meromorphic connection with poles along $D$ defined in a neighbourhood of $D$ in $X$, and where $\iso : M_{\hat{D}}\longrightarrow \mathcal{M}_{\hat{D}}$ is an isomorphism of formal connections. 
\begin{theorem}\label{CDG}
Let $X$ be a germ of smooth algebraic variety around a point $0$. Let $D$ be a germ of divisor passing through $0$. Let $\mathcal{M}$ be a germ of meromorphic connection at $0$ with poles along $D$. Let $C$ be a smooth curve passing through $0$ and not contained in any of the irreducible components of $D$. If  $(M_1, \iso_1)$  and $(M_2, \iso_2)$  are $\mathcal{M}$-marked connections such that   $$(M_1, \iso_1)_{|C}\simeq (M_2, \iso_2)_{|C}$$ then 
$(M_1, \iso_1)$  and $(M_2, \iso_2)$  are isomorphic in a neighbourhood of $0$.
\end{theorem}
Let us give  an outline of the paper. In section \ref{LRJMU}, we introduce the global variant of the moduli of Stokes torsors constructed in \cite{TeySke} suited for the proof of Theorem \ref{mainth}. We then prove Theorem 4. In section $2$, we interpret the tangent spaces and the obstruction theory for these moduli in a transcendental way via the solution complexes for connections. We then prove Theorem 5.  In section 3, we show how to reduce the proof of Theorem  \ref{mainth} to extending the good formal model of $\mathcal{M}$ across the point $0$ under study. In section 4, we show that the sought-after extension exists provided that the moduli of Stokes torsors associated to a resolution of the turning point $0$ for $\mathcal{M}$ satisfies suitable geometric conditions. Finally, we show that these geometric conditions are always satisfied when the hypothesis of Theorem \ref{mainth} are satisfied, thus concluding the proof of Theorem \ref{mainth}. The last section is devoted to the proof of Theorem \ref{boundedness}.


\subsection*{Acknowledgement}
I thank Y. André, P. Boalch, M, Brion, H. Esnault, F. Loray, C. Sabbah, T. Saito, C. Simpson and T. Mochizuki for interesting discussions and constructive remarks on a first draft of this work. I thank C. Sabbah for sharing with me the intuition that the techniques developed in this paper could be applied to the questions considered in \cite{CDG}. I thank H. Hu for stimulating exchanges on  nearby slopes. I thank N. Budur  and W. Veys for constant support during the preparation of this paper. This work has been funded by the long term structural funding-Methusalem grant of the Flemish Government. I thank KU Leuven for providing outstanding working conditions. This paper benefited from a one month stay at the Hausdorff Research Institute for Mathematics, Bonn. I thank the Hausdorff Institute for providing outstanding working conditions.

\section{Moduli of Stokes torsors. Global aspects}\label{LRJMU}
\subsection{Why moduli of Stokes torsors?}\label{why}
Let us explain in this subsection how the moduli of Stokes torsors were found to be relevant to the proof of Theorem 1. We use the notations from the introduction and work in dimension 2. We suppose that $0\in D$ lies in the smooth locus of $(\Sol \mathcal{M})_{|D}$ and $(\Sol \End\mathcal{M})_{|D}$, and we want to prove that $0$ is a good formal decomposition point for $\mathcal{M}$.\\ \indent
From a theorem of Kedlaya \cite{Kedlaya1}\cite{Kedlaya2}  and  Mochizuki \cite{Mochizuki2}\cite{Mochizuki1}, our connection $\mathcal{M}$ acquires good formal decomposition at any point after pulling-back by a suitable sequence of blow-ups above $D$. To test the validity of the conjecture \cite[15.0.5]{phdteyssier}, a natural case to consider was the  case where only one blow-up is needed. Using results of André \cite{Andre}, it was shown in \cite{LetterToSabbah} that the conjecture reduces in this case to the following 
\begin{question}
Given two good meromorphic connections $\mathcal{M}$ and  $\mathcal{N}$ with poles along the coordinate axis in $\mathds{C}^{2}$ and formally isomorphic at $0$, is it true that 
\begin{equation}\label{egalitedim}
 \dim (\mathcal{H}^1\Sol \End \mathcal{M})_0= \dim (\mathcal{H}^1\Sol \End \mathcal{N})_0  \hspace{0.5cm} ?
\end{equation}
\end{question}
It turns out that each side of \eqref{egalitedim} appeared as dimensions of moduli spaces of Stokes torsors constructed by Babbitt-Varadarajan in \cite{BV}.
These moduli were associated with germs of meromorphic connections in dimension $1$. Babbitt and Varadarajan proved that they are affine spaces. This suggested the existence of a moduli $\mathcal{X}$ with two points $P, Q\in \mathcal{X}$ such that the left-hand side of \eqref{egalitedim} would be $\dim T_P\mathcal{X}$ and the right-hand side of \eqref{egalitedim} would be $\dim T_Q\mathcal{X}$. The equality \eqref{egalitedim} would then follow from the smoothness and connectedness of the putative moduli. This is what led to \cite{TeySke}, but the question of smoothness and connectedness was left open. In the meantime, a positive answer to the above question was given by purely analytic means by C. Sabbah in \cite{SabRemar}.

\subsection{Relation with \cite{TeySke}}
In \cite{TeySke}, a moduli for local Stokes torsors  was constructed in any dimension. This moduli suffers two drawbacks in view of the proof of Theorem \ref{mainth}. First, the Stokes sheaf used in  \cite{TeySke} only makes sense at a neigbourhood of a point, whereas our situation will be global as soon as we apply Kedlaya-Mochizuki's resolution of turning points. Second, the relation between Irregularity and the tangent spaces of the moduli from \cite{TeySke} only holds in particular cases. To convert the hypothesis on Irregularity appearing in  Theorem \ref{mainth} into a geometric statement pertaining to moduli of torsors, we need to replace the Stokes sheaf $\St_{\mathcal{M}}$ of a connection $\mathcal{M}$ by a subsheaf denoted by $\St_{\mathcal{M}}^{<D}$. We will abuse terminology be also calling the torsors under $\St_{\mathcal{M}}^{<D}$ Stokes torsors. The sheaf $\St_{\mathcal{M}}^{<D}$ has the advantage of being globally defined when $\mathcal{M}$ is globally defined. Along the smooth locus of $D$, the sheaf  $\St_{\mathcal{M}}^{<D}$ is the usual Stokes sheaf. The only difference between 
$\St_{\mathcal{M}}$ and $\St_{\mathcal{M}}^{<D}$ appears at a singular point of $D$.

\subsection{Geometric setup}\label{setup}
In this subsection, we introduce basic notations. Let $X$ be a smooth complex algebraic  variety of dimension $n$. Let $D$ be a normal crossing divisor in $X$. For a quasi-coherent sheaf $\mathcal{F}$ on $X$, we denote by  $\mathcal{F}_{|D}$ the sheaf of germs of sections of $\mathcal{F}$ along $D$.  Let $D_1, \dots , D_m$ be the irreducible components of $D$. 
For $I\subset \llbracket 1, m\rrbracket $, set 
$$
D_I := \bigcap_{i\in I}D_i \text{ and }  D_I^{\circ} := D_I \setminus \displaystyle{\bigcup_{i\notin I } }  D_i
$$
Let $\rho$ be a metric on $X$. For $I\subset \llbracket 1, m\rrbracket $ and for $\epsilon>0$ small enough, put 
$$\Delta_I:=\{x\in D | \exists y \in D_I \text{ with }  \rho(x,y)<\epsilon \}$$
and
$$
\Delta_I^{\circ} := \Delta_I \setminus \displaystyle{\bigcup_{i\notin I } }  D_i
$$
The $\Delta_I^{\circ}, I\subset \llbracket 1, m\rrbracket$ form an open cover of $D$. Since  $\epsilon$ will not play any explicit role in the sequel, we will slightly abuse terminology by replacing the expression "at the cost of shrinking $\epsilon$" by "at the cost of shrinking the $\Delta_I$".

\subsection{Functions with asymptotic expansion along $D$}
For $i=1, \dots, m$, let $\tilde{X}_i\longrightarrow X$ be the real blow-up of $X$ along $D_i$.
Let $p_D : \tilde{X}\longrightarrow X$ be the fibre product of the $\tilde{X}_i$, $i=1, \dots, m$ above $X$. For every subset $A\subset D$, put $\partial A := p^{-1}_D(A)$. Let  $\iota_A : \partial A\longrightarrow \partial D$ be the canonical inclusion. \\ \indent
 Let $\mathcal{A}$ be the sheaf of  functions on $\partial D$ admitting an asymptotic expansion along $D$ \cite{Sabbahdim}. For a closed subset $Z$ in $D$, let $\mathcal{A}_{\hat{Z}}$ be the completion of $\mathcal{A}$ along the pull-back by $p_D$ of the ideal sheaf of $Z$. Put  $\mathcal{A}^{<Z}:=\Ker(\mathcal{A}\longrightarrow \mathcal{A}_{\hat{Z}})$. When $Z=D$, the sheaf $\mathcal{A}^{<D}$ can be concretely described locally as follows (see \cite[II 1.1.11]{Sabbahdim} for a proof). Let
$(x_1, \dots ,  x_n)$ be local coordinates centred at $0\in D$ such that $D$ is defined around $0$ by $x_1 \cdots x_l = 0$ for some $l\in \llbracket 1, m\rrbracket$. Then, 
the germ of $\mathcal{A}^{<D}$ at $\theta  \in \partial 0$ is given by those holomorphic functions $u$ defined over the trace 
on $X\setminus D$ of a neighbourhood $\Omega$  of $\theta$ in $\tilde{X}$, and such that for every compact $K\subset \Omega$, for every $N:=(N_1, \dots , N_l )\in \mathds{N}^l$, there exists a constant $C_{K,N }>0$ satisfying
$$
|u(x)|\leq C_{K,N} |x_1|^{N_1}  \cdots  |x_l|^{N_l}   \text{ for every $x\in K\cap (X\setminus D)$ }
$$

 \subsection{Torsors}
Let $M$ be a manifold. Let $\mathcal{G}$ be a sheaf of groups on $M$. We recall that a torsor under $\mathcal{G}$ is a sheaf $\mathcal{F}$ on $M$ endowed with a left action of $\mathcal{G}$ such that there exists a cover $\mathcal{U}$ by open subsets of $M$ such that for every $U\in \mathcal{U}$, there exists an isomorphism of sheaves $\mathcal{F}_{|U}\simeq \mathcal{G}_{|U}$ commuting with the action of $\mathcal{G}$, where $\mathcal{G}$ acts on itself by multiplication on the left. It is a standard fact that isomorphism classes of $\mathcal{G}$-torsors are in bijection with $H^{1}(M,\mathcal{G})$, the set of non abelian cohomology classes of $\mathcal{G}$. 

\subsection{Stokes torsors and the functor of relative Stokes torsors}\label{defmoduli}
Let $\mathcal{M}$ be a good meromorphic connection defined in a neighbourhood of $D$ and with poles along $D$.
We set  
$$
\partial\mathcal{M}= \mathcal{A}\otimes_{p^{-1}_D\mathcal{O}_{X|D}}p^{-1}_D\mathcal{M}
$$ 
and
$$
\partial\mathcal{M}_{\hat{D}}= \mathcal{A}_{\hat{D}}\otimes_{p^{-1}_D\mathcal{O}_{X|D}}p^{-1}_D\mathcal{M}
$$ 
Let $\mathcal{D}_X$ be the sheaf of differential operators on $X$. The sheaf $\mathcal{A}$ is endowed with an action of $p^{-1}_D\mathcal{D}_{X|D}$. Hence, so does $\partial\mathcal{M}$. We can thus form the De Rham complex of $\mathcal{M}$ with coefficients in $\mathcal{A}$ as
$$
\xymatrix{
\partial\mathcal{M} \ar[r] & p^{-1}_D\Omega^1_{X|D}\otimes_{p^{-1}_D\mathcal{O}_{X|D}}\partial\mathcal{M} \ar[r] &\cdots  \ar[r] &  p^{-1}_D\Omega^n_{X|D}\otimes_{p^{-1}_D\mathcal{O}_{X|D}}\partial\mathcal{M} 
}
$$
It is denoted by $\DR \partial\mathcal{M}$. Similarly, we denote by $\DR^{<D}\mathcal{M}$ the De Rham complex of $\mathcal{M}$ with coefficients in $\mathcal{A}^{<D}$.\\ \indent
Let $Z$ be a closed subset of $D$. Let $\St_{\mathcal{M}}^{<Z}$ be the subsheaf of  $\mathcal{H}^{0}\DR \partial\End\mathcal{M}$
of sections asymptotic to the Identity along $Z$, that is of the form $\Id +f$ where $f$ has coefficients in $\mathcal{A}^{<Z}$. The sheaf $\St_{\mathcal{M}}^{<Z}$ is a sheaf of complex unipotent algebraic groups on $\partial Z$. This is \textit{the Stokes sheaf of $\mathcal{M}$ along $Z$}.\\ \indent 
Since $\St_{\mathcal{M}}^{<Z}$ is a sheaf of complex algebraic groups, for every $R\in \mathds{C}$-alg, the sheaf of $R$-points of  $\St_{\mathcal{M}}^{<Z}$ is a well-defined sheaf of groups on $\partial Z$. It is denoted by $\St_{\mathcal{M}}^{<Z}(R)$. This is the Stokes sheaf of $\mathcal{M}$ along $Z$ relative to $R$. Torsors under $\St_{\mathcal{M}}^{<Z}(R)$ are the \textit{Stokes torsors along $Z$ relative to $R$}.
For every subset $A\subset Z$, let $H^{1}( \partial A, \St_{\mathcal{M}}^{<Z})$ be the functor 
\begin{eqnarray*}
\text{$\mathds{C}$-alg} &    \longrightarrow &  \Set\\
   R  &            \longrightarrow  &   H^{1}(\partial A, \St_{\mathcal{M}}^{<Z}(R)) 
\end{eqnarray*}
From \cite[Th. 1]{TeySke}, the functor $H^{1}( \partial P, \St_{\mathcal{M}}^{<P})$ is an affine scheme of finite type over $\mathds{C}$ for every $P\in D$. The main goal of this section to prove that  $H^{1}( \partial D, \St_{\mathcal{M}}^{<D})$ is also an affine scheme of finite type over $\mathds{C}$. Note that in the rank $2$ case, this scheme is known to be an affine space  \cite{Rare}. We start with the following
\begin{lemme}\label{nontrivialauto}
Torsors under $\St_{\mathcal{M}}^{<D}$ have no non trivial automorphisms.
\end{lemme}
\begin{proof}
It is enough to show that  $\St_{\mathcal{M}}^{<D}$-torsors above a point $P\in D$ have no non trivial automorphisms. To do this, we can suppose that $\mathcal{M}$ is unramified. Let $\mathcal{T}$ be a $\St_{\mathcal{M}}^{<D}$-torsor on $\partial P$. Let $\phi : \mathcal{T}\longrightarrow \mathcal{T}$ be an automorphism of $\St_{\mathcal{M}}^{<D}$-torsors. Since $\mathcal{A}^{<D}$ is a subsheaf of $\mathcal{A}^{<P}$, there is an injection $\iota : \St_{\mathcal{M}}^{<D}\longrightarrow \St_{\mathcal{M}}^{<P}$. To show that $\phi$ is the identity of $\mathcal{T}$ amounts to show that the push-forward $\iota_{\ast}\phi : \iota_{\ast}\mathcal{T}\longrightarrow \iota_{\ast}\mathcal{T}$ is the identity of the $\St_{\mathcal{M}}^{<P}$-torsor $\iota_{\ast}\mathcal{T}$. This last assertion is a consequence of  \cite[1.8.1]{TeySke}. This finishes the proof of lemma \ref{nontrivialauto}.
\end{proof}
As a straighforward consequence of lemma \ref{nontrivialauto}, we deduce the following
\begin{corollaire}\label{faisceau}
The presheaf of functors $R^1p_{D\ast}\St_{\mathcal{M}}^{<D}$ defined as
\begin{eqnarray*}
\Open(D) &    \longrightarrow &  \Set\\
   U  &            \longrightarrow  &   H^{1}(\partial U, \St_{\mathcal{M}}^{<D}) 
\end{eqnarray*}
is a sheaf of functors. That is, for every cover $\mathcal{U}$ of $D$ by open subsets, the first arrow in the  following diagram of pointed functors
$$
\xymatrix{
H^{1}( \partial D, \St_{\mathcal{M}}^{<D})            \ar[r] &  \prod_{U\in \mathcal{U}} H^{1}( \partial U, \St_{\mathcal{M}}^{<D}) \ar@<-.5ex>[r] \ar@<.5ex>[r] & \prod_{U,V\in \mathcal{U}} H^{1}( \partial U\cap \partial V, \St_{\mathcal{M}}^{<D}) 
}
$$
is an equalizer.
\end{corollaire}

The first goal of this section is to prove the following representability theorem
\begin{theorem}\label{rep}
The functor $H^{1}( \partial D, \St_{\mathcal{M}}^{<D})$ is representable by an affine scheme of finite type over $\mathds{C}$. 
\end{theorem}
In particular, Theorem \ref{rep} says that the sheaf of functors $R^1p_{D\ast}\St_{\mathcal{M}}^{<D}$ on $D$ is a sheaf of affine schemes of finite type over $\mathds{C}$. In order to prove Theorem \ref{mainth}, we will need to understand the geometry of the transition maps of the sheaf $R^1p_{D\ast}\St_{\mathcal{M}}^{<D}$. The second goal of this section is thus to prove Theorem \ref{cestuneimmersionfermee}.

\subsection{Representability by a scheme}\label{repbyscheme}
To prove Theorem \ref{rep}, the idea is to analyse separately the contributions coming from each stratum of $D$. 
At the cost of shrinking the $\Delta_I$, Mochizuki's local extension lemma \cite[3.17]{MochStokes} implies that $\iota_{D_I^{\circ}}^{-1}\St_{\mathcal{M}}^{<D}$-torsors extend canonically to 
$\partial\Delta_{I}^{\circ}$ for every $I\subset \llbracket 1, m\rrbracket $. Hence, there is a canonical isomorphism of functors
$$
\xymatrix{
H^{1}( \partial D_{I}^{\circ}, \St_{\mathcal{M}}^{<D})            \ar[r]^-{\sim} &  
H^{1}(\partial \Delta_{I}^{\circ}, \St_{\mathcal{M}}^{<D})   
}
$$
Applying corollary \ref{faisceau} to the open cover of $D$ formed by the $\Delta_{I}^{\circ}$, $I \subset  \llbracket 1, m\rrbracket$, we deduce that $H^{1}(\partial D, \St_{\mathcal{M}}^{<D})$ is a finite limit of the functors $H^{1}( \partial D_{I}^{\circ}, \St_{\mathcal{M}}^{<D})$, $I \subset  \llbracket 1, m\rrbracket$. Hence, Theorem  \ref{rep} is a consequence of the following

\begin{proposition}\label{smooth} 
Let $I \subset  \llbracket 1, m\rrbracket$ be non empty. Suppose that $D_I^{\circ}$ is connected and pick $P\in D_I^{\circ}$. Then, the functors $H^{1}( \partial D_I^{\circ}, \St_{\mathcal{M}}^{<D})$ and $H^{1}( \partial P, \St_{\mathcal{M}}^{<D})$ are affine schemes of finite type over $\mathds{C}$. The natural morphism
$$
\xymatrix{
H^{1}( \partial D_I^{\circ} , \St_{\mathcal{M}}^{<D})          \ar[r] & H^{1}( \partial P, \St_{\mathcal{M}}^{<D})     
}
$$
is a closed immersion.
\end{proposition}

\begin{proof} 
Let  $p_I^{\circ}$ be the restriction of $p_D$ to $\partial D_{I}^{\circ}$. Let $B$ be a ball in $D_{I}^{\circ}$. Let $U, V\subset B$ be connected open subsets such that $U\subset V$. Let $Q\in U$. From Mochizuki's extension theorem \cite[3.9]{MochStokes}\footnote{See also \cite[II 6.1]{Frob} for the case where $I$ is a singleton.}, the restriction morphisms in
$$
\xymatrix{
H^{1}( \partial V , \St_{\mathcal{M}}^{<D})          \ar[r] \ar[rd] & H^{1}( \partial U, \St_{\mathcal{M}}^{<D})    \ar[d]   \\
 &     H^{1}( \partial Q, \St_{\mathcal{M}}^{<D}) 
}
$$
are isomorphisms of functors. Hence,  $R^1p_{I\ast}^{\circ}\St_{\mathcal{M}}^{<D}$ is a local system of functors on $D_{I}^{\circ}$ in the sense of \cite{SimpsonII}, and the stalk of $R^1p_{I\ast}^{\circ}\St_{\mathcal{M}}^{<D}$ at $P$ is  $H^{1}( \partial P, \St_{\mathcal{M}}^{<D})$. Thus, $H^{1}(\partial  D_I^{\circ} , \St_{\mathcal{M}}^{<D})$ is the functor of invariants for the action of $\pi_1(D_I^{\circ} , P)$ on $H^{1}(\partial P, \St_{\mathcal{M}}^{<D})$. That is, if $(\gamma_1, \dots, \gamma_N)$ denotes a set of generators for $\pi_1(D_I^{\circ} , P)$, the following diagram of functors
$$
\xymatrix{
H^{1}( \partial D_I^{\circ}  , \St_{\mathcal{M}}^{<D})          \ar[rr] \ar[d]&& H^{1}( \partial P, \St_{\mathcal{M}}^{<D})    \ar[d]^-{(\Id,\gamma_1, \dots, \gamma_N)}   \\
H^{1}( \partial P, \St_{\mathcal{M}}^{<D})  \ar[rr]_-{\text{Diagonal}}  &&     H^{1}( \partial P, \St_{\mathcal{M}}^{<D})^{N+1}
}
$$
is cartesian. To prove Theorem \ref{rep}, we are thus left to prove that $H^{1}(\partial P, \St_{\mathcal{M}}^{<D})$ is an affine scheme of finite type over $\mathds{C}$. We argue recursively on the cardinality of $I$. If $I$ is a singleton, this is a consequence of Babbit-Varadarajan's representability theorem \cite{BV}. In general, this is a consequence of the recursion hypothesis combined with the proposition \ref{singpt} below.

\end{proof}

\begin{proposition}\label{singpt} 
Let $I \subset  \llbracket 1, m\rrbracket$ with at least two elements. Let $P\in D_I^{\circ}$. Let $i\in I$ such that the difference of any two distinct irregular values for $\mathcal{M}$ at $P$  has poles along $D_i$\footnote{Such a component exists by goodness property of the irregular values of $\mathcal{M}$.}. Then, for a small enough neighbourhood $\Delta$ of $P$ in $D_{I\setminus \{i\}}$, extension-restriction of torsors 
\begin{equation}\label{uniso}
\xymatrix{
H^{1}( \partial P, \St_{\mathcal{M}}^{<D})          \ar[r]&  
H^{1}( \partial \Delta^{\ast}, \St_{\mathcal{M}}^{<D})  
}
\end{equation}
is an isomorphism of functors, where $\Delta^{\ast}=\Delta\setminus D_i\subset D_{I\setminus \{i\}}^{\circ}$.
\end{proposition}
\begin{proof}
By Galois descent, it is enough to treat the case where $\mathcal{M}$ is unramified.
At the cost of reindexing the components of $D$, we can take local coordinates $(x_1,\dots , x_n)$ centred at $P$ such that $D$ is given by $x_1\cdots x_l = 0$ in a neighbourhood of $P$ and such that $i=l$. In particular,
$$\tilde{X}\simeq ([0, +\infty [ \times S^{1})^{l}\times \mathds{C}^{n-l}$$ 
and $p_D$ reads
$$
((r_k,\theta_k )_{k},y)\longrightarrow ((r_k e^{i\theta_k} )_{k},y)
$$
Let $\Delta$ be a small enough polydisc in $D_{I\setminus \{i\}}$ such that any $\St_{\mathcal{M}}^{<D}$-torsor on $\partial P$ extends canonically to $\partial \Delta$. The map
\eqref{uniso} is then defined as the composition map
$$
\xymatrix{
H^{1}( \partial P, \St_{\mathcal{M}}^{<D})          \ar[r]^-{\sim}&  H^{1}( \partial \Delta, \St_{\mathcal{M}}^{<D}) \ar[r] & 
H^{1}( \partial \Delta^{\ast}, \St_{\mathcal{M}}^{<D})  
}
$$
Let $j :\partial \Delta^{\ast}\longrightarrow \partial D$ be the canonical inclusion. 
\begin{lemme}\label{atorsor} 
For every $\mathcal{T}\in 
H^{1}(\partial \Delta^{\ast}, \St_{\mathcal{M}}^{<D})$, the sheaf $\iota_P^{-1} j_{\ast}\mathcal{T}$ is a $\iota_P^{-1}j_{\ast}j^{-1}\St_{\mathcal{M}}^{<D}$-torsor on $\partial P$.
\end{lemme}
\begin{proof}
Let $(\theta_0, 0)\in \partial P$. Let $I_1, \dots, I_l$ be intervals in $S^1$ such that $\theta_0 \in I_1\times  \dots \times I_l$. Set
$$
\mathcal{S}= ([0, r [ \times I_1) \times \cdots \times  ([0, r [ \times I_l )\times \Delta^{\prime}
 $$
where $r>0$ and where $\Delta^{\prime}$ is a neighbourhood of $0$ in $\mathds{C}^{n-l}$. To prove lemma \ref{atorsor}, we have to prove that $\mathcal{T}$ is the trivial $\St_{\mathcal{M}}^{<D}$-torsor on $\mathcal{S} \cap \partial\Delta^{\ast}$ for $\mathcal{S}$ small enough. For $\mathcal{S}$ small enough, we have 
$$
\mathcal{S} \cap \partial\Delta^{\ast}= (\{0\} \times I_1) \times \cdots \times (\{0\} \times I_{l-1} )\times  (]0, r [ \times I_l )\times \Delta^{\prime}
 $$
Let $a,b$ be distinct irregular values for $\mathcal{M}$ at $P$. Put $F_{a,b}:=  \Real(a-b)|z^{-\ord(a-b)}|$ and put $\ord(a-b)=-(\alpha_{ab}(1), \dots, \alpha_{ab}(l))\in \mathds{Z}^l_{\leq 0}$. By definition, the Stokes hyperplanes for $(a,b)$ are the connected components of $F_{a,b}\circ p_D=0$. They are of the form
$$
\theta_{ab}(z)-\sum_{i=1}^{l}\alpha_{ab}(i) \theta_i= \frac{\pi}{2}+ k\pi
$$
where $k\in \mathds{Z}$ and where $\theta_{ab}$ is a continuous function defined in a neigbourhood of $0$ in $\mathds{C}^n$. For $\mathcal{S}$ small enough, the open 
set $\mathcal{S} \cap \partial\Delta^{\ast}$ thus meets at most one Stokes hyperplane for $(a,b)$. From Mochizuki's splitting lemma \cite[4.1.5]{Mochizuki1}, we deduce that  $\mathcal{T}$ is trivial on $\mathcal{S} \cap \partial\Delta^{\ast}$. This finishes the proof of lemma \ref{atorsor}.
\end{proof}
 From lemma \ref{atorsor}, $\iota_P^{-1}j_{\ast}$ induces a morphism of functors
 \begin{equation}\label{inverse}
\xymatrix{
H^{1}( \partial \Delta^{\ast}, \St_{\mathcal{M}}^{<D})          \ar[r]&  
H^{1}( \partial P, \iota_P^{-1}j_{\ast}j^{-1}\St_{\mathcal{M}}^{<D})  
}
\end{equation}
 If we prove that the adjunction morphism 
\begin{equation}\label{adjmorp}
\xymatrix{
\iota_P^{-1} \St_{\mathcal{M}}^{<D}          \ar[r] &  
\iota_P^{-1}j_{\ast}j^{-1}\St_{\mathcal{M}}^{<D}   
}
\end{equation}
is an isomorphism, then \eqref{inverse} will provide us with an inverse for \eqref{uniso}. We now prove that \eqref{adjmorp} is an isomorphism. Injectivity is obvious, so we are left to prove surjectivity. This is a local statement on $\partial P$. Hence, Mochizuki's asymptotic development theorem \cite[3.2.10]{Mochizuki1} reduces the question to the case where $\mathcal{M}$ is split unramified. If $\mathcal{I}$ denotes the set of irregular values for $\mathcal{M}$ at $P$, this means that 
$$
\mathcal{M}= \bigoplus_{a\in \mathcal{I}}\mathcal{E}^{a} \otimes \mathcal{R}_a
$$
where $\mathcal{E}^{a}=(\mathcal{O}_{X}(\ast D), d-da)$ and where $\mathcal{R}_a$ is regular. Let $i_a : \mathcal{E}^{a} \otimes \mathcal{R}_a\longrightarrow \mathcal{N}$ be the canonical inclusion and let $p_a : \mathcal{N} \longrightarrow \mathcal{E}^{a} \otimes \mathcal{R}_a$ be the canonical projection. Let $I_1, \dots , I_l$ be intervals in $S^1$. Set
$$
\mathcal{S}= ([0, r [ \times I_1) \times \cdots \times  ([0, r [ \times I_l )\times \Delta^{\prime}
 $$
where $r>0$ is small enough and where $\Delta^{\prime}$ is a small enough neighbourhood of $0$ in $\mathds{C}^{n-l}$.
Sections of $\St_{\mathcal{M}}^{<D}$ on $\mathcal{S} \cap \partial D$ are automorphisms of $\mathcal{M}$ on $\mathcal{S}\cap (X\setminus D)$ of the form $\Id + f$ where $p_a f i_b=0$ unless 
\begin{equation}\label{rapidD}
e^{a-b}\in \Gamma(\mathcal{S} \cap \partial D, \mathcal{A}^{<D})
\end{equation}
Sections of $\St_{\mathcal{M}}^{<D}$  on 
$$
\mathcal{S} \cap \partial\Delta^{\ast}= (\{0\} \times I_1) \times \cdots \times (\{0\} \times I_{l-1} )\times  (]0, r [ \times I_l )\times \Delta^{\prime}
 $$
are automorphisms of $\mathcal{M}$ on $\mathcal{S}\cap (X\setminus D)$ of the form $\Id + f$ where $p_a f i_b=0$ unless 
\begin{equation}\label{rapiddecay}
e^{a-b}\in \Gamma( \mathcal{S} \cap \partial\Delta^{\ast}, \mathcal{A}^{<D})
\end{equation}
We thus have to show that for every distinct irregular values $a,b$ for $\mathcal{M}$, the conditions  \eqref{rapidD} and \eqref{rapiddecay} are equivalent for a small enough choice of $\mathcal{S}$. A change of variable reduces the problem to the case where $a-b = 1/x_1^{\alpha_1}\cdots x_l^{\alpha_l}$ where $(\alpha_1, \dots , \alpha_l)\in \mathds{N}^{l-1}\times \mathds{N}^{\ast}$. Note that condition \eqref{rapidD} trivially implies condition \eqref{rapiddecay}.
Suppose that $e^{1/x^{\alpha_1}_1\cdots x_l^{\alpha_l}}\in \Gamma(\mathcal{S} \cap \partial\Delta^{\ast}, \mathcal{A}^{<D})$. At the cost of shrinking $\mathcal{S}$, this implies that there exists a constant $C>0$ such that for every
$$
(x_1,\dots, x_n)\in  (]0, r [ \times I_1) \times \cdots \times (]0, r [  \times I_{l-1} )\times  ([\frac{r}{2}, r [ \times I_l )\times \Delta^{\prime}
$$ 
we have 
$$
|e^{1/x_1^{\alpha_1}\cdots x_l^{\alpha_l}} | \leq C|x_1|\cdots |x_{l-1}|
$$ 
Writing $x_i=(r_i, \theta_i)$ for $i=1, \dots , l-1$, this means
$$
 e^{\cos(\alpha_1 \theta_1+\cdots +  \alpha_{l} \theta_{l})/r_1^{\alpha_1}\cdots r_{l}^{\alpha_{l}}}  \leq C r_1\cdots r_{l-1} 
$$
In particular, $\alpha_i >0$ for $i=1, \dots , l-1$ and $\cos(\alpha_1 \theta_1+\cdots +  \alpha_{l} \theta_{l})<0$ for every 
$(\theta_1, \dots , \theta_l)\in I_1 \times \cdots  \times I_l$. At the cost of shrinking $\mathcal{S}$ further, there exists $c>0$ such that $\cos(\alpha_1 \theta_1+\cdots +  \alpha_{l} \theta_{l})<-c $ on $I_1 \times \cdots  \times I_l$. Then, we have 
$$|e^{1/x_1^{\alpha_1}\cdots x_l^{\alpha_l}}  | \leq  e^{-c/|x_1|^{\alpha_1}\cdots |x_l|^{\alpha_l}}$$ on $\mathcal{S}$. Since
 $\alpha_i >0$ for $i=1, \dots , l$, we deduce that \eqref{rapiddecay} holds, which proves the equivalence between conditions \eqref{rapidD} and \eqref{rapiddecay}. This finishes the proof of proposition \ref{uniso}, and thus the proof of Theorem \ref{rep}.
 
 \end{proof}
 
 We store the following immediate corollary from proposition \ref{smooth} for later use.
 
 \begin{corollaire}\label{deUaVfermee}
Let $I \subset  \llbracket 1, m\rrbracket$. Let $U\subset V \subset D_I^{\circ}$ be non empty open subsets in $D_I^{\circ}$ such that $V$ is connected. Then, the natural morphism
$$
\xymatrix{
H^{1}( \partial V , \St_{\mathcal{M}}^{<D})          \ar[r] & H^{1}( \partial U, \St_{\mathcal{M}}^{<D})     
}
$$ 
is a closed immersion. 
 \end{corollaire}
 \begin{proof}
 Choose a point $P\in U$. Then, there is a factorization
 $$
 \xymatrix{
H^{1}(\partial V, \St_{\mathcal{M}}^{<D})   \ar[rd] \ar[r] &  H^{1}(\partial U, \St_{\mathcal{M}}^{<D}) \ar[d]  \\
&    H^{1}(\partial P, \St_{\mathcal{M}}^{<D})
}
 $$
From proposition \ref{smooth}, the diagonal arrow is a closed immersion between affine schemes. Hence, the horizontal arrow is a closed immersion.
 \end{proof}
 
\subsection{Passing from one stratum to an other stratum is a closed immersion}
The next proposition is the technical core of this paper.
\begin{proposition}\label{verygoodclosedimmersion}
Let $I \subset  \llbracket 1, m\rrbracket$ non empty. Let $i\in I$. Let $P\in D_I^{\circ}$. Then, for a small enough  neighbourhood $\Delta$  of $P$ in $D_{i}$, the morphism of schemes  
\begin{equation}\label{cestiso}
\xymatrix{
H^{1}( \partial P, \St_{\mathcal{M}}^{<D})          \ar[r] &  
H^{1}( \partial \Delta^{\ast}, \St_{\mathcal{M}}^{<D})  
}
\end{equation}
is a closed immersion, where $\Delta^{\ast}=\Delta\setminus \bigcup_{j\in I\setminus \{i\}} D_j\subset D_{i}^{\circ}$.
\end{proposition}
\begin{proof}
Note that both functors appearing in  \eqref{cestiso} are affine schemes as a consequence of  proposition \ref{smooth}. Let $j :\partial \Delta^{\ast}\longrightarrow \partial D$ be the canonical inclusion. The sheaf of algebraic groups $\iota_{D_i}^{-1}\St_{\mathcal{M}}^{<D}$ is distinguished in $\St_{\mathcal{M}}^{<D_i}$. We thus have an exact sequence of sheaves of algebraic groups on $\partial D_i$
$$
\xymatrix{
1 \ar[r] &  \iota_{D_i}^{-1}   \St_{\mathcal{M}}^{<D}  \ar[r]^-{\iota} &   \St_{\mathcal{M}}^{<D_i} \ar[r] &    \mathcal{Q}  \ar[r] &   1
}
$$
There is an adjunction morphism
\begin{equation}\label{adjfact}
\xymatrix{
\iota_P^{-1}  \St_{\mathcal{M}}^{<D_i} \ar[r] &  
\iota_P^{-1}j_{\ast}j^{-1}\St_{\mathcal{M}}^{<D_i}=\iota_P^{-1}j_{\ast}j_{}^{-1}\St_{\mathcal{M}}^{<D}
}
\end{equation}
Hence, there is a factorization
\begin{equation}\label{factorization}
\xymatrix{
H^{1}(\partial P, \St_{\mathcal{M}}^{<D})   \ar[rd] \ar[r]^{\iota_{\ast}} &  H^{1}(\partial P, \St_{\mathcal{M}}^{<D_i}) \ar[d]  \\
&    H^{1}(\partial\Delta^{\ast}, \St_{\mathcal{M}}^{<D})
}
\end{equation}
From a similar argument to that in proposition \ref{singpt}, the adjunction morphism \eqref{adjfact} is an isomorphism of sheaves on $\partial P$. Hence, the vertical arrow in \eqref{factorization} is an isomorphism of functors. From proposition \ref{smooth}, we deduce first that $H^{1}(\partial P, \St_{\mathcal{M}}^{<D_i})$  is an affine scheme of finite type over $\mathds{C}$. Second, we deduce that to prove proposition \ref{verygoodclosedimmersion}, it is enough to prove that 
$$
\xymatrix{
\iota_\ast : H^{1}(\partial P, \St_{\mathcal{M}}^{<D})   \ar[r] &  H^{1}(\partial P,  \St_{\mathcal{M}}^{<D_i})
}
$$
is a closed immersion. 
From \cite[I.2]{JFrenkel}, there is an exact sequence of pointed functors
\begin{equation}\label{exactsequJFRenkel}
\xymatrix{
  H^0( \partial P, \mathcal{Q} ) \ar[r] &    H^1( \partial P, \St_{\mathcal{M}}^{<D}  )\ar[r]^-{\iota_{\ast}} &  H^1( \partial P,  \St_{\mathcal{M}}^{<D_i} ) \ar[r] &    H^1( \partial P,  \mathcal{Q}   )
}
\end{equation}
Let us prove that $H^0( \partial P, \mathcal{Q} )$ is trivial. The complex of sheaves 
$$
\xymatrix{
  \St_{\mathcal{M}}^{<D_i}  \ar[r]&   \partial\End \mathcal{M}_{\hat{D}}     \ar[r]& \partial\End \mathcal{M}_{\hat{D_i}}   
}
$$
induces a sequence of sheaves
\begin{equation}\label{exQ}
\xymatrix{
  \mathcal{Q}  \ar@{^{(}->}[r]&   \partial\End \mathcal{M}_{ \hat{D}}     \ar[r]&\partial \End \mathcal{M}_{ \hat{D_i}}   
}
\end{equation}
By applying $p_{D\ast}$ and then looking at the germs at $P$, we deduce from \cite[p44]{Sabbahdim} the following sequence
\begin{equation}\label{exQ2}
\xymatrix{
0\ar[r] &  H^0( \partial P, \mathcal{Q} ) \ar[r]&   \End \mathcal{M}_{\hat{D}, P}     \ar[r]& \End \mathcal{M}_{\hat{D_i}, P}   
}
\end{equation}
By flatness of $\End \mathcal{M}$ over $\mathcal{O}_X$, the second map in \eqref{exQ2} is injective. Hence,  $H^0( \partial P, \mathcal{Q} )$ is trivial. Thus, the following diagram of functors 
\begin{equation}\label{carteiota}
\xymatrix{
 H^1( \partial P, \St_{\mathcal{M}}^{<D}  ) \ar[d]_-{\iota_{\ast}} \ar[r]&   \ast     \ar[d]    \\
  H^{1}(\partial P, \St_{\mathcal{M}}^{<D_i})     \ar[r]&  H^1( \partial P, \mathcal{Q}  ) 
}
\end{equation}
is cartesian, where $\ast$ denotes the trivial $\mathcal{Q}$-torsor. If we knew that $H^1( \partial P, \mathcal{Q})$ is a scheme, we would directly obtain that $\iota_{\ast}$ is a closed immersion. This question does not seem to follow from the use of skeletons in \cite{TeySke}. We will circumvent this problem with a group theoretic argument. \\ \indent
The Stokes hyperplanes of $\mathcal{M}$ above $P$ form a finite family of closed subsets in $\partial P$. Note that the associated stratification of $\partial P$ is such that the restriction of $\St_{\mathcal{M}}^{<D_i}$ to each stratum is locally constant. Hence, the same argument as in \cite[1.9.1]{TeySke} applies. In particular, there exists a cover $\mathcal{U}$ of $\partial P$ by open subsets such that the  morphism of affine schemes 
\begin{equation}\label{goodZtoH}
\xymatrix{
 Z^1( \mathcal{U},\St_{\mathcal{M}}^{<D_i})  \ar[r]&   H^{1}(\partial P, \St_{\mathcal{M}}^{<D_i}) 
 }   
\end{equation}
is surjective at the level of $R$-points for every $R\in \mathds{C}$-alg. From \cite[2.7.3]{BV}, the morphism \eqref{goodZtoH} admits a section. Composing this section with
$$
\xymatrix{
 Z^1( \mathcal{U}, \St_{\mathcal{M}}^{<D_i})  \ar[r]&    Z^1( \mathcal{U},  \mathcal{Q} )   
 }
$$
gives rise to a commutative triangle of functors
\begin{equation}\label{fact}
\xymatrix{
    H^{1}(\partial P, \St_{\mathcal{M}}^{<D_i})   \ar[r] \ar[rd]   & H^1( \partial P, \mathcal{Q})       \\
    &    Z^1( \mathcal{U},  \mathcal{Q} )  \ar[u]
  }
\end{equation}
The algebraic group
$$
G_{\mathcal{U}} :=  \displaystyle{\prod}_{U\in \mathcal{U} }  \Gamma(U, \mathcal{Q})
$$
acts on $Z^1( \mathcal{U},  \mathcal{Q} )$. Let 
\begin{equation}\label{mutrivialcocycle}
G_{\mathcal{U}}\longrightarrow Z^1( \mathcal{U},  \mathcal{Q} )
\end{equation}
 be the morphism of schemes obtained by restricting the action of $G_{\mathcal{U} }$ to the trivial cocycle. Since $H^0( \partial P, \mathcal{Q} )\simeq 0$, the morphism \eqref{mutrivialcocycle} is a monomorphism. There is a commutative diagram
\begin{equation}\label{incompletediagram}
\xymatrix{
 H^1( \partial P, \St_{\mathcal{M}}^{<D}  ) \ar[d]_-{\iota_{\ast}} &   G_{\mathcal{U}} \ar[d]  \ar[r]  &   \ast  \ar[d]\\
  H^{1}(\partial P, \St_{\mathcal{M}}^{<D_i})     \ar[r]& Z^1( \mathcal{U},  \mathcal{Q} ) \ar[r] & H^1( \partial P, \mathcal{Q}  ) 
}
\end{equation}
 We would like to reduce the problem of proving that $ \iota_{\ast}$ is a closed immersion to the problem of proving that \eqref{mutrivialcocycle} is a closed immersion. To do this, we would like to fill the left diagram in \eqref{incompletediagram} into a cartesian square. Note that the right square in \eqref{incompletediagram} may not be cartesian since there may be cocycles in $Z^1( \mathcal{U},  \mathcal{Q})$ that are cohomologous to the trivial cocycle only after passing to a refinement of $\mathcal{U}$. To treat this problem, we argue by using the universal torsor under $\St_{\mathcal{M}}^{<D}$on $\partial P$. \\ \indent
Let $\mathcal{T}^{\univ}$ be the universal torsor under $\St_{\mathcal{M}}^{<D}$ on $\partial P$. Let $A$ be the ring of functions of  $H^1( \partial P, \St_{\mathcal{M}}^{<D})$. From the commutativity of \eqref{fact}, the image $\gamma$ of $\mathcal{T}^{\univ}$ in $Z^1( \mathcal{U},  \mathcal{Q}(A) )$ induces the trivial 
 $\mathcal{Q}(A)$-torsor. Hence, there exists a refinement $\mathcal{V}$ of $\mathcal{U}$ such that $\gamma_{|\mathcal{V}}$ is cohomologous to the trivial cocycle, that is, such that $\gamma_{|\mathcal{V}}$ lies in the image of 
$G_{\mathcal{V}}(A)\longrightarrow Z^1( \mathcal{V},  \mathcal{Q}(A))$. Hence, there is a commutative square 
\begin{equation}\label{completediagram}
 \xymatrix{
 H^1( \partial P, \St_{\mathcal{M}}^{<D}  )\ar[r]  \ar[d]_-{\iota_{\ast}} &   G_{\mathcal{V}} \ar[d] \\
  H^{1}(\partial P, \St_{\mathcal{M}}^{<D_i})     \ar[r]& Z^1( \mathcal{V},  \mathcal{Q} ) 
}
 \end{equation}
This square is cartesian. Indeed, let $F$  be the fibre product of $H^{1}(\partial P, \St_{\mathcal{M}}^{<D_i})$  with $G_{\mathcal{V}}$ over $Z^1( \mathcal{V},  \mathcal{Q})$. By definition, there is a commutative diagram of functors
 \begin{equation}\label{trianglesubrunctor}
 \xymatrix{
 H^1( \partial P, \St_{\mathcal{M}}^{<D}  )\ar[r]  \ar[rd]_{\iota_{\ast}} &   F       \ar[d] \\
                                                                       & H^{1}(\partial P, \St_{\mathcal{M}}^{<D_i})  
}
 \end{equation}
Since the right vertical arrow in \eqref{completediagram} is a monomorphism,  $F$ is a sub-functor of $H^{1}(\partial P, \St_{\mathcal{M}}^{<D_i})$. Hence, all maps in \eqref{trianglesubrunctor} are inclusions of functors. We are thus left to prove that $F$ is a sub-functor of $H^1( \partial P, \St_{\mathcal{M}}^{<D})$. This is an immediate consequence of the fact that $H^1( \partial P, \St_{\mathcal{M}}^{<D})$ is the functor of torsors $\mathcal{T}\in H^{1}(\partial P, \St_{\mathcal{M}}^{<D_i})$ inducing the trivial $ \mathcal{Q}$-torsor. \\ \indent 
Hence, to prove that  $\iota_{\ast}$ is a closed immersion, we are left to show that \eqref{mutrivialcocycle} for $\mathcal{V}$ is a closed immersion. 
%
%
From the general theory of algebraic group actions, the map \eqref{mutrivialcocycle} factors as 
$$
\xymatrix{
G_{\mathcal{V}}\ar[r]^-{\alpha}    &  \mathcal{O} \ar[r]^-{\beta}  &  Z^1( \mathcal{V},  \mathcal{Q} )
}
$$
where $\alpha$ is faithfully flat, where $\mathcal{O}$ is the orbit of the trivial cocycle under $G_{\mathcal{V}}$ and where $\beta$ is an immersion of schemes. Since smoothness is a local property for the fppf topology \cite[05B5]{SPDescent}, the smoothness of $G_{\mathcal{V}}$ implies that $\mathcal{O}$  is smooth. By definition, $\alpha$ is an isomorphism at the level of $\mathds{C}$-points. Hence, $\alpha$ is an isomorphism of varieties. We are thus left to show that $\mathcal{O}$ is closed in $Z^1( \mathcal{V},  \mathcal{Q} )$. It is enough to show that $\mathcal{O}$ is closed in $Z^1( \mathcal{V},  \mathcal{Q} )^{\reduit}$. From Kostant-Rosenlicht theorem \cite[I 4.10]{BorelAlgGp}, it is enough to show that $G_{\mathcal{V}}$ is a unipotent algebraic group, which is a consequence of the fact that the Stokes sheaves are sheaves of unipotent algebraic groups. This concludes the proof of  
proposition \ref{verygoodclosedimmersion}.

\end{proof}

\subsection{Proof of Theorem \ref{cestuneimmersionfermee}}

Let $U\subset V\subset D$ be non empty open subsets in $D$ such that $V$ is connected. We want to show that the natural morphism
\begin{equation}\label{closedimmersion0}
\xymatrix{
H^{1}( \partial V, \St_{\mathcal{M}}^{<D})            \ar[r] &  
H^{1}(\partial U, \St_{\mathcal{M}}^{<D})   
}
\end{equation}
is a closed immersion of affine schemes of finite type over $\mathds{C}$.
Let $A$ be the set of open subsets $U^{\prime}$ in $V$ containing $U$ and such that the natural morphism 
$$
H^{1}( \partial U^{\prime}, \St_{\mathcal{M}}^{<D}) \longrightarrow  H^{1}(\partial U, \St_{\mathcal{M}}^{<D})
$$ 
is a closed immersion. We 
want to show that $A$ contains $V$. Note that $A$ is not empty since it contains $U$. Let $
B$ be a subset of  $A$ which is totally ordered for the inclusion. Let $R$ be the ring of functions of $H^{1}( \partial U, \St_{\mathcal{M}}
^{<D})$. For $U^{\prime}\in B$, let $\mathcal{I}_{U^{\prime}}$ be the ideal of functions of $H^{1}( \partial 
U^{\prime}, \St_{\mathcal{M}}^{<D})$ in $H^{1}( \partial U, \St_{\mathcal{M}}^{<D})$. By assumption on 
$B$, the family of ideals $(\mathcal{I}_{U^{\prime}})_{U^{\prime}\in B}$ is totally ordered for the 
inclusion. Hence, $\mathcal{I}:=\bigcup_{U^{\prime}\in \mathcal{I}_{U^{\prime}}}$ is an ideal in $R$. Since $R$ is 
noetherian, there exists $U^{\prime}_0\in B$ such that $\mathcal{I}=\mathcal{I}_{U^{\prime}_0}$. In particular, $\mathcal{I}_{U^{\prime}}=\mathcal{I}_{U^{\prime}_0}$ for every $U^{\prime}\in B$ containing $U^{\prime}_0$. Set $V^{\prime}:=\bigcup_{U^{\prime}\in B}U^{\prime}$. From lemma \ref{faisceau}, we deduce 
\begin{align*}
H^{1}( \partial V^{\prime}, \St_{\mathcal{M}}^{<D})         &   \simeq  \lim_{U^{\prime}\in B} 
H^{1}( \partial U^{\prime}, \St_{\mathcal{M}}^{<D})   \\
 &\simeq   \lim_{U^{\prime}\in B, U^{\prime}_0\subset U^{\prime}} 
H^{1}( \partial U^{\prime}, \St_{\mathcal{M}}^{<D}) \\
                &    \simeq    H^{1}( \partial U^{\prime}_0, \St_{\mathcal{M}}^{<D}) 
\end{align*}
Thus, $V^{\prime}\in A$. From Zorn lemma, we deduce that  $A$ admits a maximal element $W$.
If $W$ is closed in $V$, then we have $W=V$ by connectedness of $V$. Suppose now that $W$ is not closed in $V$. Let $P\in \overline{W}\setminus W$ and let $B$ be a small ball in $V$ containing $P$ and such that $H^{1}( \partial 
B, \St_{\mathcal{M}}^{<D}) \longrightarrow  H^{1}(\partial P, \St_{\mathcal{M}}^{<D})$ is an isomorphism. Set $W^{\prime}:= W\cup B\subset V$. We are going to show that $W^{\prime}\in A$, which contradicts the fact that $W$ is maximal in $A$. From the factorization
\begin{equation}\label{factori}
\xymatrix{
H^{1}( \partial W^{\prime}, \St_{\mathcal{M}}^{<D})  \ar[r] &    H^{1}( \partial W, \St_{\mathcal{M}}^{<D})        \ar[r] &  
H^{1}(\partial U, \St_{\mathcal{M}}^{<D})   
}
\end{equation}
we are left to show that the first arrow in \eqref{factori} is a closed immersion. From lemma \ref{faisceau}, the following diagram
\begin{equation}\label{reducB}
\xymatrix{
H^{1}( \partial W^{\prime}, \St_{\mathcal{M}}^{<D}) \ar[r]  \ar[d] & 
H^{1}( \partial B, \St_{\mathcal{M}}^{<D})\ar[d]    \\
H^{1}( \partial W, \St_{\mathcal{M}}^{<D})  \ar[r]  &  
H^{1}( \partial (W\cap B), \St_{\mathcal{M}}^{<D})  
}
\end{equation}
is cartesian. Hence, it is enough to show that the right vertical arrow in \eqref{reducB} is a closed immersion. Let 
$(P_n)_{n\in \mathds{N}}$ be a sequence of points in $W$ converging to $P$. Since $W$ is open, the sequence  $(P_n)_{n\in \mathds{N}}$ can be supposed to lie in some $D_i^{\circ}$ for $i\in \llbracket 1, m\rrbracket$. Let $\Delta \subset B$ be a small enough neighbourhood of $P$ in $D_i$. Set $\Delta^{\ast}=\Delta\setminus \bigcup_{j\in I\setminus \{i\}} D_j\subset D_{i}^{\circ}$. From our choice for $i$, the open set $W\cap \Delta^{\ast}$ is not empty. We have the following commutative diagram
\begin{equation}\label{finaldiag}
\xymatrix{
H^{1}( \partial B, \St_{\mathcal{M}}^{<D}) \ar[r]  \ar[d] & 
H^{1}( \partial \Delta^{\ast}, \St_{\mathcal{M}}^{<D})\ar[d]    \\
H^{1}( \partial (W\cap B), \St_{\mathcal{M}}^{<D})   \ar[r]  &  
H^{1}( \partial (W\cap \Delta^{\ast}), \St_{\mathcal{M}}^{<D})  
}
\end{equation}
From proposition \ref{verygoodclosedimmersion}, the top horizontal arrow in \eqref{finaldiag} is a closed immersion.       From corollary \ref{deUaVfermee}, the right vertical arrow in \eqref{finaldiag} is a closed immersion. Hence, the left vertical arrow in \eqref{finaldiag} is a closed immersion. Hence, $W^{\prime}\in A$, which contradicts the fact that $W$ is maximal in $A$. Thus, $W=V\in A$, which finishes the proof of Theorem  \ref{cestuneimmersionfermee}.

\section{Stokes torsors and marked connections}
\subsection{Notations}\label{notmarkedconn} In this section, let $X$ be a smooth complex algebraic variety. Let $D$ be a normal crossing divisor in $X$. Let $\mathcal{M}$ be a good meromorphic flat bundle on $X$ with poles along $D$.
\subsection{Definition of marked connections and relation with Stokes torsors} \label{Standmarked}
Let us recall that a \textit{$\mathcal{M}$-marked connection} is the data of a couple $(M, \iso)$ where $M$ is a germ of meromorphic connection with poles along $D$ defined in a neighbourhood of $D$ in $X$, and where $\iso : M_{\hat{D}}\longrightarrow \mathcal{M}_{\hat{D}}$ is an isomorphism of formal connections. We denote by $\Isom_{\iso}(M, \mathcal{M})$  the $\St_{\mathcal{M}}^{<D}(\mathds{C})$-torsor of isomorphisms between $\partial M$ and $\partial \mathcal{M}$ which are  asymptotic to $\iso$ along $D$. The proof of the following statement was suggested to me by T. Mochizuki. I thank him for kindly sharing it. When $D$ is smooth, it was known to Malgrange \cite{Mal2MathPhy}. See also \cite[II 6.3]{Frob}.
\begin{lemme}\label{bijectionettangent}
The map associating to every isomorphism class of  $\mathcal{M}$-marked connection  $(M, \iso)$ the $\St_{\mathcal{M}}^{<D}(\mathds{C})$-torsor $\Isom_{\iso}(M, \mathcal{M})$ is bijective.
\end{lemme}
\begin{proof}
Let us construct an inverse. Take $\mathcal{T}\in \St_{\mathcal{M}}^{<D}(\mathds{C})$ and let $g=(g_{ij})$ be a cocycle for $\mathcal{T}$ associated to a cover $(U_i)_{i\in I}$ of $\partial D$. 
Let $\mathcal{L}$ be the Stokes filtered local system on $\partial D$ associated to $\mathcal{M}$. Set $\mathcal{L}_i := \mathcal{L}_{| U_i}$. Then, $g$ allows to glue the 
$\mathcal{L}_{i}$ into a Stokes filtered local system $\mathcal{L}_{\mathcal{T}}$ on $\partial D$ independent of the choice of $g$. From the irregular Riemann-Hilbert correspondence \cite[4.11]{MochStokes}, $\mathcal{L}_{\mathcal{T}}$ is the Stokes filtered local system associated to a unique (up to isomorphism) good meromorphic connection $\mathcal{M}_{\mathcal{T}}$ defined in a neighbourhood of $D$ and with poles along $D$. By construction, the isomorphism $\mathcal{L}_{\mathcal{T}|U_i} \longrightarrow \mathcal{L}_{| U_i}$ corresponds to an isomorphism 
$\partial \mathcal{M}_{\mathcal{T}|U_i}\longrightarrow \partial\mathcal{M}_{| U_i} $. We thus obtain a formal isomorphism 
$\iso_i : \partial \mathcal{M}_{\mathcal{T},\hat{D} |U_i, }\longrightarrow \partial\mathcal{M}_{\hat{D}| U_i} $. On $U_{ij}$, the discrepancy between $\iso_i$  and $\iso_j$ is measured by the asymptotic of $g_{ij}$ along $D$. By definition, this asymptotic is $\Id$. Hence, the $\iso_i$ glue into a globally defined isomorphism $ \partial \mathcal{M}_{\mathcal{T},\hat{D}} \longrightarrow \partial \mathcal{M}_{\hat{D}}$. Applying $p_{D\ast}$ thus yields an isomorphism $\iso : \mathcal{M}_{\mathcal{T},\hat{D}} \longrightarrow  \mathcal{M}_{\hat{D}}$. It is then standard to  check that the map $\mathcal{T}\longrightarrow (\mathcal{M}_{\mathcal{T}}, \iso)$ is the sought-after inverse.
\end{proof}

\subsection{Proof of Theorem \ref{CDG}}
We are now in position to prove Theorem \ref{CDG}. Let $X$ be a germ of smooth algebraic variety around a point $0$. Let $D$ be a germ of divisor passing through $0$. Let $\mathcal{M}$ be a germ of meromorphic connection at $0$ with poles along $D$. Let $C$ be a smooth curve passing through $0$ and not contained in any of the irreducible components of $D$. Let  $(M_1, \iso_1)$  and $(M_2, \iso_2)$  be $\mathcal{M}$-marked connections such that   
$$(M_1, \iso_1)_{|C}\simeq (M_2, \iso_2)_{|C}$$
We want to show that $(M_1, \iso_1)$  and $(M_2, \iso_2)$  are isomorphic in a neighbourhood of $0$. Let $\pi : Y\longrightarrow X$ be a resolution of  turning points for  $\mathcal{M}$ around $0$. Such a resolution exists by works of Kedlaya \cite{Kedlaya2}  and  Mochizuki \cite{Mochizuki1}. Set $E:=\pi^{-1}(D)$. At the cost of blowing up further, we can suppose that the strict transform $C^{\prime}$ of $C$ is transverse to $E$ at a point $P$ in the smooth locus of $E$. Note that $E$ is connected. From lemma \ref{bijectionettangent}, the $\pi^{+}\mathcal{M}$-marked connections $(\pi^{+}M_1, \pi^{+}\iso_1)$ and $(\pi^{+}M_2, \pi^{+}\iso_2)$ define two $\mathds{C}$-points of  $H^{1}(\partial E, \St_{\pi^{+}\mathcal{M}}^{<E})$. Since 
$$(\pi_{+}\pi^{+}M_i, \pi_{+}\pi^{+}\iso_i)\simeq (M_i, \iso_i)$$
for $i=1,2$, it is enough to show $(\pi^{+}M_1, \pi^{+}\iso_1)\simeq (\pi^{+}M_2, \pi^{+}\iso_2)$. By assumption, 
\begin{align*}
(\pi^{+}M_1, \pi^{+}\iso_1)_{|C^{\prime}} & \simeq    (M_1, \iso_1)_{|C}  \\
                                                       & \simeq (M_2, \iso_2)_{|C}    \\
                                                       & \simeq (\pi^{+}M_2, \pi^{+}\iso_2)_{|C^{\prime}}
\end{align*}
In particular, the image of $(\pi^{+}M_1, \pi^{+}\iso_1)$ and $(\pi^{+}M_2, \pi^{+}\iso_2)$  by the restriction map 
\begin{equation}\label{preuveth3}
\xymatrix{
H^{1}( \partial E, \St_{\pi^{+}\mathcal{M}}^{<E})            \ar[r] &  
H^{1}(\partial P, \St_{\pi^{+}\mathcal{M}}^{<E})   
}
\end{equation}
are the same. From Theorem \ref{cestuneimmersionfermee}, the map \eqref{preuveth3} is a closed immersion. Hence, $(\pi^{+}M_1, \pi^{+}\iso_1)\simeq (\pi^{+}M_2, \pi^{+}\iso_2)$, which concludes the proof of Theorem \ref{CDG}.

\subsection{Obstruction theory and tangent space}
We use the notations from \ref{notmarkedconn}. Let us compute the obstruction theory of $H^{1}(\partial D, \St_{\mathcal{M}}^{<D})$ at a point $\mathcal{T}_0\in H^{1}(\partial D, \St_{\mathcal{N}}^{<D}(\mathds{C}))$. We fix a morphism of infinitesimal extensions of $\mathds{C}$-algebras  
$$
R^{\prime}\longrightarrow R\longrightarrow \mathds{C}, \hspace{20pt} I:=\Ker R^{\prime}\longrightarrow R
$$ 
such that $I$ is annihilated by $\Ker R^{\prime}\longrightarrow \mathds{C}$. In particular, $I^2=0$ and $I$ is endowed with a structure of $\mathds{C}$-vector space, which we suppose to be finite dimensional. Let $\mathcal{T}\in H^{1}(\partial D, \St_{\mathcal{M}}^{<D}(R))$ lifting $\mathcal{T}_0$.  Choose a cover $\mathcal{U}=(U_i)_{i\in K}$ of $\partial D$ such that 
$\mathcal{T}$ comes from a cocycle $g=(g_{ij})_{i,j\in K}$. Set $L_i(R):=\Lie \St_{\mathcal{M}}^{<D}(R)_{|U_i}$. The identifications   
\begin{eqnarray*}
L_i(R)_{|U_{ij}}    &     \overset{\sim}{\longrightarrow}  & L_j(R)_{|U_{ij}}  \\
   M   &            \longrightarrow  & g_{ij}^{-1}M g_{ij}
\end{eqnarray*}
allow to glue the $L_i(R)$ into a sheaf of $R$-Lie algebras over $\partial D$ denoted by $\Lie \St_{\mathcal{M}}^{<D}(R)^{\mathcal{T}}$ and depending only on $\mathcal{T}$ and not on $g$. For $t=(t_{ijk})\in \check{C}^{2}(\mathcal{U}, \Lie \St_{\mathcal{M}}^{<D}(R)^{\mathcal{T}})$, we denote by $s_{ijk}$ the unique representative of $t_{ijk}$ in $\Gamma(U_{ijk}, L_i(R))$. Then
\begin{align*}
(dt)_{ijkl}&=t_{jkl}-t_{ikl}+t_{ijl}-t_{ijk}\\            
&=[g_{ij}s_{jkl}g_{ij}^{-1}-s_{ikl}+s_{ijl}-s_{ijk}]
\end{align*}
We have the following
\begin{lemme}\label{obs}
There exists
$$
\ob(\mathcal{T})\in I\otimes_{\mathds{C}}\check{H}^{2}(\partial D, \Lie \St_{\mathcal{M}}^{<D}(\mathds{C})^{\mathcal{T}_0})
$$
such that $\ob(\mathcal{T})=0$ if and only if $\mathcal{T}$ lifts to $H^{1}(\partial D, \St_{\mathcal{M}}^{<D}(R^{\prime}))$.
\end{lemme}
\begin{proof}
For every $i,j \in K$, let $h_{ij}\in \Gamma(U_{ij}, \St_{\mathcal{M}}^{<D}(R^{\prime}))$ be an arbitrary lift of $g_{ij}$ to $R^{\prime}$. We can always choose the $h_{ij}$ to satisfy $h_{ii}=\Id$ and $h_{ij}h_{ji}=\Id$.
Since $\Lie \St_{\mathcal{M}}^{<D}(R^{\prime})$ is locally free, 
$$
I \cdot \Lie \St_{\mathcal{M}}^{<D}(R^{\prime})\simeq I \otimes_{R^{\prime}} \Lie \St_{\mathcal{M}}^{<D}(R^{\prime})\simeq I \otimes_{\mathds{C}} \Lie \St_{\mathcal{M}}^{<D}(\mathds{C})
$$
We will use both descriptions without mention. We set 
$$
s_{ijk}:=h_{ij}h_{jk} h_{ki}-\Id \in \Gamma(U_{ijk}, I \cdot \Lie \St_{\mathcal{M}}^{<D}(R^{\prime}))
$$
We see $s_{ijk}$ as a section of $I \otimes_{\mathds{C}} L_i(\mathds{C})$ over $U_{ijk}$ and denote by $[s_{ijk}]$ its class in $I \otimes_{\mathds{C}}  \Lie \St_{\mathcal{M}}^{<D}(\mathds{C})^{\mathcal{T}_0}$.
We want to prove that the $[s_{ijk}]$ define a cocycle. As seen above, this amounts to prove the following equality in $\Gamma(U_{ijk}, I\otimes_{\mathds{C}}\Lie \St_{\mathcal{M}}^{<D}(\mathds{C}))$
\begin{equation}\label{equa}
g_{ij}(0)s_{jkl}g_{ij}^{-1}(0)-s_{ikl}+s_{ijl}-s_{ijk} =0
\end{equation} 
Where $g_{ij}(0)$ is the image of $g_{ij}$ by $R\longrightarrow \mathds{C}$.
We have
\begin{align*}
 g_{ij}(0)s_{jkl}g_{ij}^{-1}(0) & = h_{ij} h_{jk} h_{kl} h_{lj} h_{ji} -\Id\\
       & =  (h_{ij} h_{jk} - h_{ik}+ h_{ik})h_{kl} h_{lj} h_{ji} -\Id\\
       & = (h_{ij} h_{jk} - h_{ik})g_{kl}(0) g_{lj}(0) g_{ji}(0) + h_{ik}h_{kl} h_{lj} h_{ji} -\Id\\
       & = (h_{ij} h_{jk} - h_{ik})g_{ki}(0) + h_{ik}h_{kl} h_{lj} h_{ji}-\Id\\
        & =(h_{ij} h_{jk} - h_{ik})h_{ki} + h_{ik}h_{kl} h_{lj} h_{ji}-\Id\\
          & =h_{ij} h_{jk}h_{ki}  + h_{ik}h_{kl} h_{lj} h_{ji}-2\Id
\end{align*}
We now see how the second term of the last line above interacts with the second term of the left-hand side of \eqref{equa}.
\begin{align*}
 h_{ik}h_{kl} h_{lj} h_{ji}-s_{ikl} & =  h_{ik}h_{kl} h_{lj} h_{ji}-h_{ik} h_{kl} h_{li}+\Id\\
  & =  h_{ik}h_{kl}( h_{lj} h_{ji}-h_{li})+\Id\\
  &= g_{ik}(0)g_{kl}(0)( h_{lj} h_{ji}-h_{li})+\Id \\
  & =g_{il}(0)( h_{lj} h_{ji}-h_{li})+\Id \\
  & =h_{il} h_{lj} h_{ji}
\end{align*}
Hence, 
\begin{align*}
g_{ij}(0)s_{jkl}g_{ij}^{-1}(0)-s_{ikl}+s_{ijl}-s_{ijk}  & = h_{il} h_{lj} h_{ji}+h_{ij} h_{jl} h_{li}  -2\Id \\
&= (h_{ij} h_{jl} h_{li} )^{-1}+h_{ij} h_{jl} h_{li}  -2\Id \\
  &=  (h_{ij} h_{jl} h_{li} )^{-1}((h_{ij} h_{jl} h_{li} )^{2}-2h_{ij} h_{jl} h_{li} +\Id) \\   
  & = (h_{ij} h_{jl} h_{li} )^{-1}s^2_{ijl}\\
  & =0
\end{align*}
where the last equality comes from $I^{2}=0$. Hence, the $[s_{ijk}]$ define a cocycle of $I\otimes_{\mathds{C}}\Lie \St_{\mathcal{M}}^{<D}(\mathds{C})^{\mathcal{T}_0}$. An other choice of lift gives rise to  homologous cocycles. We denote by $\ob(\mathcal{T})$ the class of $([s_{ijk}])_{ijk}$ in 
$\check{H}^{2}(\partial D , I\otimes_{\mathds{C}}\Lie \St_{\mathcal{M}}^{<D}(\mathds{C})^{\mathcal{T}_0})$.
It is standard to check that $\ob(\mathcal{T})$ has the sought-after property.
\end{proof}

\begin{corollaire}\label{obetIrr}
Let $(M, \iso)$ be a $\mathcal{M}$-marked connection. Then, the space \newline  $H^{2}(D , \Irr^{\ast}_D\End M))$ is
an obstruction theory for $H^{1}(\partial D, \St_{\mathcal{M}}^{<D})$ at $\Isom_{\iso}(M, \mathcal{M})$.
\end{corollaire}
\begin{proof}
 Set $\mathcal{T}:= \Isom_{\iso}(M, \mathcal{M})$. As observed in \cite[5.2]{TeySke}, the canonical identification 
$$
\xymatrix{
\mathcal{H}^0 \DR^{<D}\End \mathcal{M}\ar[r]^-{\sim} & \Lie \St_{\mathcal{M}}^{<D}(\mathds{C})^{\mathcal{T}}
}
$$ 
induces 
\begin{align*}
\check{H}^{i}(\partial D,  \Lie \St_{\mathcal{M}}^{<D}(\mathds{C})^{\mathcal{T}})&\simeq 
H^{i}(\partial D,  \Lie \St_{\mathcal{M}}^{<D}(\mathds{C})^{\mathcal{T}})\\
&\simeq  H^i( \partial D, \mathcal{H}^0 \DR^{<D}\End M)  \\
 & \simeq 
      H^i(\partial D, \DR^{<D}\End M)  \\
      & \simeq H^{i}(D , \Irr^{\ast}_D\End M)
\end{align*}
The second identification comes from the fact \cite[Prop. 1]{HienInv} that $\DR^{<D}\End \mathcal{M}$ is concentrated in degree $0$. The third identification comes from \cite[2.2]{SabRemar}. Then, corollary \ref{obetIrr} follows from lemma \ref{obs}.
\end{proof}

Reasoning exactly as in \cite[5.2.1]{TeySke}, we prove the following 

\begin{lemme}\label{espaceTan}
For every $\mathcal{M}$-marked connection $(M, \iso)$, the tangent space of $H^{1}(\partial D, \St_{\mathcal{M}}^{<D})$ at $(M, \iso)$ identifies canonically to $H^1(D,  \Irr_D^{\ast} \End M)$.
\end{lemme}

\section{Reduction of Theorem \ref{mainth}  to extending the formal model}

\subsection{Reduction to the dimension 2 case}

In this subsection, we reduce the proof of Theorem \ref{mainth} 
 to the dimension $2$ case. The main tool is André's goodness criterion \cite[3.4.3]{Andre} in terms of Newton polygons. This reduction does not seem superfluous. Of crucial importance for the sequel of the proof  will be indeed the fact that when $X$ is an algebraic surface and $D$ a smooth divisor in $X$, then for every point $0\in D$ and every meromorphic connection $\mathcal{M}$ on $X$ with poles along $D$, the formal model of $\mathcal{M}$ \textit{splits on a small enough punctured disc around $0$}. This fact 
is specific to dimension $2$, since it pertains to the property that turning points in dimension  $2$ are isolated.

\begin{lemme}\label{reductionDim2}
The converse inclusion in Theorem \ref{mainth} is true in any dimension if it is true in dimension $2$.
\end{lemme}
\begin{proof}
Take $n>2$. We argue recursively by supposing that Theorem \ref{mainth} holds in dimension strictly less than $n$ and we prove that Theorem \ref{mainth} holds in dimension $n$. Let $X$ be a smooth complex algebraic variety of dimension $n$. Let $D$ be a smooth divisor in $X$. Let $\mathcal{M}$ be an algebraic meromorphic connection on $X$ with poles along $D$. Let $0\in D$ and suppose that $\Irr_D^{\ast}\mathcal{M}$ and  $\Irr_D^{\ast}\End\mathcal{M} $ are local systems in a neighbourhood of $0$. If $j : X\setminus D\longrightarrow X$ and $i: D\longrightarrow X$ are the canonical inclusions, we have a distinguished triangle 
$$
\xymatrix{
j_{!}L   \ar[r] & \Sol \mathcal{M} \ar[r] & i_{\ast}\Irr_{D}^{\ast} \mathcal{M}
}
$$
where $L$ is a local system on the complement of $D$. Hence, the characteristic cycle of $\Sol \mathcal{M}$  is supported on the union of $T^{\ast}_X X$ with  $T^{\ast}_D X$. From a theorem of Kashiwara and Schapira \cite[11.3.3]{KS}, so does the characteristic cycle of $\mathcal{M}$. Hence, any smooth hypersurface transverse to $D$ and passing through $0$ is non characteristic with respect to $\mathcal{M}$ in a neighbourhood of $0$. Let us choose such a hypersurface $Z$ and let $i_Z : Z \longrightarrow X$ be the canonical inclusion. From \cite[3.4.3]{Andre}, the turning locus of $\mathcal{M}$ is a closed subset of $D$ which is either empty or purely of codimension $1$ in $D$. Since $n>2$, the hypersurface $Z$ can consequently be chosen such that $\mathcal{M}$ and $\End \mathcal{M}$  have good formal decomposition generically along $Z\cap D$. The connection $i_Z^{+}\mathcal{M}$ is a meromorphic connection with poles along $Z\cap D$. It satisfies the hypothesis of Theorem \ref{mainth} at the point $0$. Indeed
by Kashiwara's restriction theorem \cite{TheseKashiwara}, 
$$
\Irr_{Z\cap D}^{\ast}i_Z^{+}\mathcal{M}= (\Sol i_Z^{+}\mathcal{M})_{|Z\cap D} \simeq (\Sol \mathcal{M})_{|Z\cap D} 
$$
and similarly for $\End \mathcal{M}$. Hence, $\Irr_{Z\cap D}^{\ast}i_Z^{+}\mathcal{M}$   and $\Irr_{Z\cap D}^{\ast}\End i_Z^{+}\mathcal{M}$ are local systems in a neighbourhood of $0$ in $Z\cap D$. By recursion hypothesis, $i_Z^{+}\mathcal{M}$ is good at $0$. In particular, the Newton polygon of $i_Z^{+}\mathcal{M}$ at $0$ (which is also the Newton polygon of $\mathcal{M}$ at $0$) is the generic Newton polygon of 
$i_Z^{+}\mathcal{M}$ along $Z\cap D$. From our choice for $Z$, 
the generic Newton polygon of 
$i_Z^{+}\mathcal{M}$ along $Z\cap D$ is the generic Newton polygon of $\mathcal{M}$ along $D$. Hence, the Newton polygon of $\mathcal{M}$ at $0$ is the generic Newton polygon of $\mathcal{M}$ along $D$, and similarly with  $\End\mathcal{M}$. By a theorem of André 
\cite[3.4.1]{Andre}, we deduce that $\mathcal{M}$  has good formal decomposition at $0$, which proves lemma \ref{reductionDim2}.

\end{proof}

\subsection{Setup and recollections}\label{extformalmodel}
From now on, we restrict the situation to dimension $2$. We use coordinates $(x,y)$ on $\mathds{A}^2_{\mathds{C}}$ and set $D_x:=\{y=0\},  D_y:=\{x=0\}$. Let $D$ be a neighbourhood of $0$ in $D_x$ and let $\mathds{C}[D]$ be the coordinate ring of $D$. Set $D^{\ast}:=D\setminus \{0\}$. \\ \indent
Let $\mathcal{M}$  be an algebraic meromorphic flat bundle on a neighbourhood of $D$ in  $\mathds{A}^2_{\mathds{C}}$ with poles along $D$. In algebraic terms, $\mathcal{M}_{\hat{D}}$ defines a $\mathds{C}[D]((y))$-differential module. At the cost of shrinking $D$ if necessary, we can suppose that the restriction $\mathcal{M}^{\ast}$ of $\mathcal{M}$ to a neighbourhood of $D^{\ast}$ has good formal decomposition at every point of $D^{\ast}$. \\ \indent
There is a ramification $v=y^{1/d}$, $d\geq 1$ and a finite Galois extension $L/\mathds{C}(x)$ such that the set $\mathcal{I}$ of generic irregular values for $\mathcal{M}$ lies in $\Frac L(v)$. If $p: D_L\longrightarrow D$ is the normalization of $D$ in $L$, the generic irregular values of $\mathcal{M}$ are thus meromorphic functions on $D_L\times \mathds{A}^1_v$. We have 
\begin{equation}\label{LT}
L((v))\otimes \mathcal{M}\simeq  \bigoplus_{a\in \mathcal{I}}\mathcal{E}^{a} \otimes \mathcal{R}_a
\end{equation}
where the $\mathcal{R}_a$ are regular. Following \cite[3.2.4]{Andre}, we recall the following
\begin{definition}
We say that \textit{$\mathcal{M}$ is semi-stable at $P\in D$} if
\begin{enumerate}
\item We have $\mathcal{I}\subset \mathds{C}[D_L]_P((v))$.
\item The decomposition \eqref{LT} descends to $\mathds{C}[D_L]_P((v))\otimes \mathcal{M}$. 
\end{enumerate}
\end{definition}
In this definition, $\mathds{C}[D_L]_P$ denotes the localization of $\mathds{C}[D_L]$ above $P$. This is a semi-local ring. Let $\pi_a \in L((v))\otimes \End\mathcal{M}$ be the projector on the factor  $\mathcal{E}^{a} \otimes \mathcal{R}_a$. As explained in \cite[3.2.2]{Andre}, the point $P$ is stable if and only if the generic irregular values of $\mathcal{M}$ and the coefficients of the $\pi_a$ in a basis of  $\End\mathcal{M}$ belong to $\mathds{C}[D_L]_P((v))$. Since $\mathcal{M}$ has good formal decomposition at any point of $D^{\ast}$, the generic irregular values of $\mathcal{M}$ and the coefficients of the $\pi_a$ in a basis of  $\End\mathcal{M}$ belong to $\mathds{C}[D_L]_P((v))$ for every $P\in D^{\ast}$. Hence, they belong $\mathds{C}[D^{\ast}_L]((v))$ where $D^{\ast}_L :=D\setminus p^{-1}(0)$. Thus
\begin{equation}\label{LTglobal}
\mathds{C}[D^{\ast}_L]((v))\otimes \mathcal{M}\simeq \mathds{C}[D^{\ast}_L]((v)) \otimes \mathcal{N}^{\ast}_L  
\end{equation}
where 
$$
\mathcal{N}^{\ast}_L =\bigoplus_{a\in \mathcal{I}}\mathcal{E}^{a} \otimes \mathcal{R}_a 
$$ 
is a germ of meromorphic connection defined on a neighbourhood of $D_L^{\ast}$ in $D_L \times \mathds{A}^1_v$ and with poles along $D_L^{\ast}$.
The action of 
$$
\Gal(L/\mathds{C}(x))\times \mathds{Z}/d\mathds{Z}
$$ 
on the left-hand side of \eqref{LTglobal} induces an action on $\mathcal{N}^{\ast}_L$. Taking the invariants yields a meromorphic flat bundle $\mathcal{N}^{\ast}$ defined on a neighbourhood $\Omega$ of $D^{\ast}$ in $\mathds{A}^2_{\mathds{C}}$. 
By Galois descent, \eqref{LTglobal} descents to an isomorphism $\iso^{\ast}$ between the formalizations of $\mathcal{M}^{\ast}$ and $\mathcal{N}^{\ast}$ along $D^{\ast}$.
\subsection{Reduction to the problem of extending the formal model}
The goal of this subsection is to show that Theorem \ref{mainth} reduces to prove that the $\mathcal{M}^{\ast}$-marked connection $(\mathcal{N}^{\ast}, \iso^{\ast} )$ defined in \ref{extformalmodel} extends into a $\mathcal{M}$-marked connection in a neighbourhood of $0$. To do this, we need three preliminary lemmas. The notations and constructions from \ref{extformalmodel} are in use.
\begin{lemme}\label{splitextension}
Suppose that $\mathcal{N}^{\ast}$ extends into a meromorphic flat bundle $\mathcal{N}$ defined in a neighbourhood of $D$ in $\mathds{A}^2_{\mathds{C}}$ and with poles along $D$. Then, $\mathcal{N}$ is semi-stable at $0$.
\end{lemme}
\begin{proof}
It is enough to treat the case where $K=\mathds{C}(x)$ and $d=1$. In that case, discussion \ref{extformalmodel}  shows that on a neighbourhood $\Omega$ of $D^{\ast}$ in $\mathds{A}^2_{\mathds{C}}$, we have
$$
\mathcal{N}^{\ast}=\bigoplus_{a\in \mathcal{I}}\mathcal{N}^{\ast}_a
$$
 where 
$\mathcal{N}^{\ast}_a$ is a meromorphic connection on $\Omega$ with poles along $D^{\ast}$ and with single irregular value $a$. The open $D\times \mathds{A}^1_{\mathds{C}}$ retracts on the small neighbourhood on which $\mathcal{N}$  is defined. Since $\mathcal{N}$  is smooth away from $D$, we deduce that $\mathcal{N}$  extends canonically into a meromorphic connection on $D\times \mathds{A}^1_{\mathds{C}}$  with poles along $D$. \\ \indent
Let $a\in \mathcal{I}$. The restriction of the projector $\pi_a$ to the complement of $D^{\ast}$ in $\Omega$ is a flat section of $\End \mathcal{N}$. Since $D^{\ast}\times \mathds{A}^1_{\mathds{C}}$ retracts on $\Omega$, parallel transport allows to extend $\pi_a$ canonically to $D^{\ast}\times \mathds{A}^1_{\mathds{C}}$. We still denote by $\pi_a$ this extension. Hence, $\mathcal{N}^{\ast}_a$ extends  into a meromorphic connection on $D^{\ast}\times \mathds{A}^1_{\mathds{C}}$  with poles along $D^{\ast}$. 
Let $\gamma$ be a small loop in $\Omega$ going around the axis $D_y$. By assumption, the monodromy of  $\mathcal{N}$ along $\gamma$ is trivial. Thus, $\pi_a$ is invariant under the monodromy of  $\End\mathcal{N}$ along $\gamma$. Hence, $\pi_a$ extends canonically to $(D\times \mathds{A}^1_{\mathds{C}})\setminus \{0\}$. By Hartog's property, it extends further into a section $\varpi_a$ of 
$\End \mathcal{N}$ on $D\times \mathds{A}^1_{\mathds{C}}$.\\ \indent
Set $\mathcal{N}_a :=\varpi_a(\mathcal{N})\subset \mathcal{N}$ for every $a\in \mathcal{I}$. We have $\varpi_a^{2}=\varpi_a$ and $\sum_{a\in \mathcal{I}}\varpi_a = \Id_{\mathcal{N}}$ because these equalities hold on a non empty open set. 
Hence,  $\mathcal{N} =\oplus_{a\in \mathcal{I}} \mathcal{N}_a$. Since $\varpi_a$ is flat, the connection on $\mathcal{N}$ preserves each $\mathcal{N}_a$. Let us prove that the $\mathcal{N}_a$ are locally free as $\mathcal{O}_{D\times \mathds{A}^1_{\mathds{C}}}(\ast D)$-modules.
\\ \indent
Let $E$ 
be a Deligne-Malgrange lattice \cite{Reseaucan} for 
$\mathcal{N}$. Since we work in dimension 2, we know from 
\cite[3.3.2]{Reseaucan} that $E$ is a vector bundle. 
We observe that $\varpi_a$ stabilizes $E$ away from $0$. By Hartog's property, we deduce that $\varpi_a$ stabilizes $E$. Hence, $\varpi_a(E)$ is a direct factor of $E$. So $\varpi_a(E)$ is  a vector bundle. Thus, 
$$
\mathcal{N}_a =\varpi_a(\mathcal{N}) =\varpi_a(E(\ast D))= (\varpi_a(E))(\ast D)
$$
 is a locally free $\mathcal{O}_{D\times \mathds{A}^1_{\mathds{C}}}(\ast D)$-module of finite rank with connection extending $\mathcal{N}^{\ast}_a$. To prove lemma  \ref{splitextension}, we are thus left to consider the case where $\mathcal{I}=\{a\}$.\\ \indent
If $\mathcal{I}=\{a\}$, then  \cite[3.3.1]{Andre} implies $a\in \mathds{C}[D]((y))$. Hence, $\mathcal{R}:=\mathcal{E}^{-a}\otimes \mathcal{N}_{\hat{D}}$ is a formal meromorphic connection with poles along $D$. By assumption, $\mathcal{R}$ is generically regular along $D$. From \cite[4.1]{Del}, we deduce that  $\mathcal{R}$ is regular. Hence, $\mathcal{N}_{\hat{D}}=\mathcal{E}^a\otimes \mathcal{R}$ with $\mathcal{R}$ regular, which concludes the proof of lemma \ref{splitextension}.
\end{proof}

\begin{lemme}\label{trueifquasisplit}
Let $\mathcal{N}$ be a meromorphic flat connection with poles along $D$. We suppose that  $\mathcal{N}$  is semi-stable at $0$ and that $\Irr_D^{\ast}\mathcal{N}$ and $\Irr_D^{\ast}\End\mathcal{N}$ are local systems in a neighbourhood of $0$. Then, $\mathcal{N}$ has good formal decomposition at $0$.
\end{lemme}
\begin{proof}
Let $\mathcal{I}$ be the set of irregular values of $\mathcal{N}$ at $0$. 
There is a ramification $v=y^{1/d}$, $d\geq 1$ and a finite Galois extension $L/\mathds{C}(x)$ such that $\mathcal{I} \subset L((v))$. Let $D_L\longrightarrow D$ be the normalization of $D$ in $L$. At the cost of shrinking $D$, we can suppose that every point of $D$ is semi-stable for $\mathcal{N}$. Hence, $\mathcal{I} \subset \mathds{C}[D_L]((v))$ and
$$
\mathds{C}[D_L]((v))\otimes \mathcal{N}=\bigoplus_{a\in \mathcal{I}}\mathcal{E}^{a} \otimes \mathcal{R}_a
$$
where  the connections $\mathcal{R}_a$ are regular. As seen in the proof of 
lemma \ref{reductionDim2}, the assumption on  $\Irr_D^{\ast}$ implies that  any smooth curve transverse to $D$ is non 
characteristic for $\mathcal{N}$. Taking the axis $D_y$  yields  
$$
\dim 
\mathcal{H}^1\Irr_{0}^{\ast} \mathcal{N}_{|D_y} =\dim 
(\mathcal{H}^1\Irr_{D}^{\ast} \mathcal{N})_0= \sum_{a\in \mathcal{I}} 
(\ord_y a)  \rk \mathcal{R}_a
$$ 
On the other hand, choose a point $P\in  D_L$ above $0$. 
Then, the irregular values of 
$\mathcal{N}_{|D_y}$ are the $a(P)$, $a\in \mathcal{I}$. Thus, 
$$
\mathcal{H}^1\Irr_{0}^{\ast} \mathcal{N}_{|D_y} =\sum_{a\in \mathcal{I}}\ord_y a(P)\rk \mathcal{R}_a
$$
Hence, $\ord_y a(P)=\ord_y a$ for every $a\in \mathcal{I}$. In 
particular, the coefficient function of the highest power of $1/v$ 
contributing to $a\in \mathcal{I}$ does not vanish at $P$. Arguing similarly for $\End\mathcal{N}$, we obtain that $\mathcal{N}$ has good formal decomposition at $0$.
\end{proof}

\begin{lemme}\label{transferthyp}
Suppose that $\Irr_D^{\ast}\mathcal{M}$ is a local system. For every $\mathcal{M}$-marked connection $(\mathcal{N}, \iso)$, the complex  $\Irr_D^{\ast}\mathcal{N}$ is a  local system.
\end{lemme}
\begin{proof}
From \cite{Mehbgro}, the complex $\Irr_{D}^{\ast} \mathcal{N}$ is perverse. To prove that it is a local system, it is thus enough to prove that the local Euler Poincaré characteristic $\chi(D, \Irr_D^{\ast}\mathcal{N}) : D\longrightarrow \mathds{Z}$ of $\Irr_{D}^{\ast} \mathcal{N}$ is constant. From the local index theorem \cite{Kajap}\cite{Bouka}, the local Euler Poincaré characteristic of $\Irr_{D}^{\ast} \mathcal{N}$ depends only on the characteristic cycle of $\mathcal{N}$. Since the characteristic cycle of $\mathcal{N}$ depends only on $\mathcal{N}$ via $\mathcal{N}_{\hat{D}}$, we have
$$
\chi(D, \Irr_D^{\ast}\mathcal{N})=\chi(D, \Irr_D^{\ast}\mathcal{M})
$$
By assumption, $\chi(D, \Irr_D^{\ast}\mathcal{M})$ is constant. Hence, $\chi(D, \Irr_D^{\ast}\mathcal{N})$ is constant, which finishes the proof of lemma \ref{transferthyp}.
\end{proof}

\begin{proposition}\label{sietendbon}
Let $D$ be an open neighbourhood of $0$ in an hyperplane of $\mathds{A}^2_{\mathds{C}}$. Let $\mathcal{M}$  be an algebraic meromorphic flat bundle on a neighbourhood of $D$ with poles along $D$. Set $D^{\ast}=D\setminus \{0\}$ and let $\mathcal{M}^{\ast}$ be the restriction of $\mathcal{M}$ to a neighbourhood of $D^{\ast}$. Let  $(\mathcal{N}^{\ast}, \iso^{\ast} )$ be the $\mathcal{M}^{\ast}$-marked connection constructed in 
\ref{extformalmodel}. Suppose that $\Irr_D^{\ast}\mathcal{M}$ and $\Irr_D^{\ast}\End\mathcal{M}$ are local systems in a neighbourhood of $0$. Then, if $(\mathcal{N}^{\ast}, \iso^{\ast} )$ extends into a $\mathcal{M}$-marked connection, $\mathcal{M}$  has good formal decomposition at  $0$.
\end{proposition}
\begin{proof}
Let $(\mathcal{N}, \iso)$ be a $\mathcal{M}$-marked connection extending $(\mathcal{N}^{\ast}, \iso^{\ast} )$.
From lemma \ref{splitextension}, the extension $\mathcal{N}$ is semi-stable at $0$.
From lemma \ref{transferthyp}, we know that $\Irr_D^{\ast}\mathcal{N}$ and $\Irr_D^{\ast}\End\mathcal{N}$ are local systems in a neighbourhood of $0$. From lemma \ref{trueifquasisplit}, we deduce that $\mathcal{N}$ has good formal decomposition at  $0$. Hence, so does $\mathcal{M}$.
\end{proof}

\section{Extension via moduli of Stokes torsors}\label{extension}

\subsection{A geometric extension criterion}\label{extensioncriterion}
In this subsection, we relate the moduli of Stokes torsors to the problem of extending marked connections. Let $D$ be an open subset of a hyperplane in $\mathds{A}^2_{\mathds{C}}$. Pick $P\in D$. Set $D^{\ast}:=D\setminus \{P\}$. Let $\mathcal{M}$  be an algebraic meromorphic flat bundle in a neighbourhood $U$ of $D$ in $\mathds{A}^2_{\mathds{C}}$ and with poles along $D$. Let $\mathcal{M}^{\ast}$ be the restriction of $\mathcal{M}$ to $U\setminus \{P\}$. Let $\pi : Y\longrightarrow \mathds{A}^2_{\mathds{C}}$ be a resolution of the turning point $P$ for  $\mathcal{M}$. Such a resolution exists by works of Kedlaya \cite{Kedlaya1}  and  Mochizuki \cite{Mochizuki2}. Let $\Delta$ be an open disc of $D$ containing $P$. Set $\Delta^{\ast}=\Delta\setminus\{P\}$. Set $E:=\pi^{-1}(\Delta)$ and pick $Q\in \Delta^{\ast}$. Let 
$$
\xymatrix{
\Phi : H^{1}(\partial E, \St_{\pi^{+}\mathcal{M}}^{<E})            \ar[r] & H^{1}(\partial Q, \St_{\mathcal{M}}^{<\Delta})       
}
$$
be the restriction morphism of Stokes torsors.
\begin{lemme}\label{exte}
 Let $(\mathcal{N}^{\ast}, \iso^{\ast})$ be a $\mathcal{M}^{\ast}$-marked connection such that $(\mathcal{N}_{Q}^{\ast}, \iso_{Q}^{\ast})$ lies in the image of $\Phi$. Then, $(\mathcal{N}^{\ast}, \iso^{\ast})$ extends into an $\mathcal{M}$-marked connection.
\end{lemme}
\begin{proof}

From lemma \ref{bijectionettangent}, any $\mathds{C}$-point of $H^{1}(\partial E, \St_{\pi^{+}\mathcal{M}}^{<E})$ comes from a unique $\pi^{+}\mathcal{M}$-marked connection. Hence, there exists $(\mathcal{N}^{\prime}, \iso^{\prime})\in H^{1}(\partial E, \St_{\pi^{+}\mathcal{M}}^{<E})$ such that $\Phi(\mathcal{N}^{\prime}, \iso^{\prime})=(\mathcal{N}_{P}^{\ast}, \iso_{P}^{\ast})$. From \cite[3.6-4]{Mehbsmf}, the $\mathcal{D}$-module $\mathcal{N}:=\pi_{+}\mathcal{N}^{\prime}$ is a meromorphic connection defined in a neighbourhood of $\Delta$ and and with poles along $\Delta$. By flat base change
\begin{align*}
\mathcal{N}_{\hat{\Delta}}& \simeq \mathcal{O}_{\hat{\mathds{A}^2_{\mathds{C}} | \Delta}}\otimes R\pi_{\ast}(\mathcal{D}_{X\rightarrow \mathds{A}^2_{\mathds{C}}}\otimes\mathcal{N}^{\prime})\\
      &\simeq   R\pi_{\ast}(\mathcal{O}_{\hat{X | E}}\otimes \mathcal{D}_{X\rightarrow \mathds{A}^2_{\mathds{C}}}\otimes\mathcal{N}^{\prime}) \\
      & \simeq R\pi_{\ast}( \mathcal{D}_{X\rightarrow \mathds{A}^2_{\mathds{C}}}\otimes\mathcal{N}^{\prime}_{\hat{E}}) \\
      &\simeq \pi_{+}\mathcal{N}^{\prime}_{\hat{E}}
\end{align*}
and similarly $\mathcal{M}_{\hat{\Delta}}\simeq \pi_{+}(\pi^{+}\mathcal{M})_{\hat{E}}$. Hence, $\iso := \pi_+ \iso^{\prime}$ defines an isomorphism between $\mathcal{N}_{\hat{\Delta}}$ and $\mathcal{M}_{\hat{\Delta}}$. So $(\mathcal{N}, \iso )$ is a $\mathcal{M}$-marked connection in a neighbourhood of $\Delta$. By definition, the germ of $(\mathcal{N}, \iso )$ at $Q$ is $(\mathcal{N}_{Q}^{\ast}, \iso_{Q}^{\ast})$. Since $R^{1}p_{\Delta\ast} \St^{<\Delta^{\ast}}_{\mathcal{M}}$ is a local system on $\Delta^{\ast}$, we deduce  
$$
(\mathcal{N}, \iso )_{|\Delta^{\ast}}=(\mathcal{N}^{\ast}, \iso^{\ast})_{|\Delta^{\ast}}
$$
Hence, the gluing of $(\mathcal{N}, \iso )$ with $(\mathcal{N}^{\ast}, \iso^{\ast})$ provides the sought-after extension of $(\mathcal{N}^{\ast}, \iso^{\ast})$ into an $\mathcal{M}$-marked connection. So lemma \ref{exte} is proved.
\end{proof}

Let us now give a sufficient condition for the surjectivity of $\Phi$ in terms of the irregularity complex.

\begin{proposition}\label{smoothproper}
With the notations from \ref{extensioncriterion}, suppose furthermore that the perverse complex $\Irr_D^{\ast}\End\mathcal{M}$ is a local system on $\Delta$. Then $\Phi$ induces an isomorphism between 
each irreducible component of $H^{1}(\partial E, 
\St_{\pi^{+}\mathcal{M}}^{<E})$ and $H^{1}(\partial Q, 
\St_{\mathcal{M}}^{<\Delta}) $.
\end{proposition}
\begin{proof}
From \cite{BV}, we know that $H^{1}(\partial Q, \St_{\mathcal{M}}^{<\Delta})$ is an affine space. Since affine spaces in characteristic $0$ have no non trivial finite étale covers, it is enough to prove that $\Phi$ is finite étale. From Theorem \ref{cestuneimmersionfermee}, the morphism $\Phi$ is a closed immersion. We are thus left to show that $\Phi$ is étale. \\ \indent
Etale morphisms between smooth schemes of finite type over $\mathds{C}$ are those morphisms inducing isomorphisms on the tangent spaces. Hence, we are left to prove that $H^{1}(\partial E,    \St_{\pi^{+}\mathcal{M}}^{<E})$ is smooth and that $\Phi$ induces isomorphisms on the tangent spaces. Let
$(M, \iso)$ be a $\pi^{+}\mathcal{M}$-marked connection. From corollary \ref{obetIrr}, an obstruction theory to lifting infinitesimally the Stokes torsor of $(M, \iso)$ is given by 
\begin{equation}\label{van}
H^{2}(E , \Irr^{\ast}_E \End M)  \simeq 
H^2(  \Delta ,  \Irr_{D}^{\ast} \pi_+\End M    ) \simeq 0
\end{equation}
The first identification expresses the  compatibility of irregularity with proper push-forward. Furthermore, from lemma \ref{transferthyp} applied to the $\End\mathcal{M}$-marked connection $(\pi_+\End M   , \pi_+ \iso)$, the perverse complex $\Irr_{D}^{\ast} \pi_+\End M $   
is a local system concentrated in degree 1. This implies the vanishing \eqref{van}. Hence, $H^{1}(\partial E,    \St_{\pi^{+}\mathcal{M}}^{<E})$ is smooth at $(M, \iso)$. From lemma \ref{bijectionettangent}, any $\mathds{C}$-point of  $H^{1}(\partial E,    \St_{\pi^{+}\mathcal{M}}^{<E})$ is of the form $(M, \iso)$. Thus, $H^{1}(\partial E,    \St_{\pi^{+}\mathcal{M}}^{<E})$ is smooth. Furthermore, we have a commutative diagram
$$
\xymatrix{
T_{(M, \iso)} H^{1}(\partial E, \St_{\pi^{+}\mathcal{M}}^{<E}) \ar[r] \ar[d]_-{\wr}  & T_{(M_Q , \iso_Q)} H^{1}(\partial Q, \St_{\mathcal{M}}^{<\Delta})  \ar[d]^-{\wr}  \\
H^1( E ,  \Irr_{E}^{\ast} \End M      ) \ar[r] \ar[d]_-{\wr}  & (\mathcal{H}^1\Irr_{D}^{\ast} \End M)_Q  \ar[d]^-{|}  \\
H^1(  \Delta ,  \Irr_{D}^{\ast} \pi_+\End M    ) \ar[r] \ar[d]_-{\wr}&  (\mathcal{H}^1\Irr_{D}^{\ast} \End M  )_Q  \ar[d]^{|}\\
H^0(  \Delta ,  \mathcal{H}^1\Irr_{D}^{\ast} \pi_+\End M     ) \ar[r]&   (\mathcal{H}^1\Irr_{D}^{\ast} \End M  )_Q 
}
$$
The first vertical maps are isomorphisms by lemma \ref{espaceTan}. As already seen,  $\Irr_{D}^{\ast} \pi_+\End M$  is a local system concentrated in degree $1$. Hence, the last vertical and the  bottom arrows are isomorphisms. Thus, the tangent map of $\Phi$ at $(M, \iso)$ is an isomorphism. This finishes the proof of proposition \ref{smoothproper}.

\end{proof}

\subsection{Proof of Theorem 1}
Let $X$ be a smooth complex algebraic variety. Let $D$ be a smooth divisor in $X$. Let $\mathcal{M}$ be an algebraic meromorphic connection with poles along $D$.  \\ \indent
We first prove the direct inclusion in Theorem 1. Suppose that $\mathcal{M}$ has good formal decomposition at a closed point $P\in D$. Since the good formal decomposition locus of $\mathcal{M}$ is open in $D$ \cite{Andre}, we can suppose at the cost of restricting the situation that $\mathcal{M}$ has good formal decomposition along $D$. By Mebkhout's theorem \cite{Mehbgro}, the complexes $\Irr_{D}^{\ast} \mathcal{M}$ and $\Irr_{D}^{\ast} \End\mathcal{M}$ are perverse. To prove that they are local systems on $D$, it is thus enough to prove that their local Euler Poincaré characteristic is constant. From the local index theorem \cite{Kajap}\cite{Bouka}, the local Euler Poincaré characteristic of $\Irr_{D}^{\ast} \mathcal{M}$ depends only on the characteristic cycle of $\mathcal{M}$. Since the characteristic cycle of $\mathcal{M}$ depends only on $\mathcal{M}$ via $\mathcal{M}_{\hat{D}}$, we are reduced to treat the case where $\mathcal{M}=\mathcal{E}^a\otimes \mathcal{R}$ where $a\in \mathcal{O}_X(\ast D)$ is good and where $\mathcal{R}$ is a regular singular meromorphic connection with poles along $D$. 
Since $\Irr_{D}^{\ast}$ is exact, we can suppose further that the rank of $\mathcal{R}$ is one. In that case, a standard computation shows that the characteristic cycle of $\mathcal{M}$ is supported on the union of $T_X^{\ast}X$ with $T_D^{\ast}X$. Hence, any smooth transverse curve to $D$ is non-characteristic for $\mathcal{M}$. Let $P\in D$ and let $C$ be a smooth transverse curve to $D$ passing through $P$. From \cite{TheseKashiwara}, we have
$$
(\Irr_{D}^{\ast} \mathcal{M})_P \simeq \Irr_{P}^{\ast} \mathcal{M}_{|C}  \simeq \mathds{C}^{\ord_D a}[-1]
$$
Hence, the local Euler-Poincaré characteristic of $\Irr_{D}^{\ast} \mathcal{M}$ is constant and similarly for $\Irr_{D}^{\ast} \End\mathcal{M}$. This finishes the proof of the direct inclusion in Theorem 1.\\ \indent
We now prove the converse inclusion in Theorem 1. From lemma \ref{reductionDim2}, we can suppose that $X$ is a surface. Let $P\in D$ such that $\Irr_{D}^{\ast} \mathcal{M}$ and $\Irr_{D}^{\ast} \End\mathcal{M}$ are local systems in a neighbourhood of $P$ in $D$. At the cost of taking local coordinates around $P$, we can suppose that $D$ is an open subset of a  hyperplane in $\mathds{A}^2_{\mathds{C}}$. Put $D^{\ast}:=D\setminus \{P\}$. Let $\mathcal{M}^{\ast}$ be the restriction of $\mathcal{M}$ to a small neighbourhood of $D^{\ast}$ in $X$. Let $(\mathcal{N}^{\ast}, \iso^{\ast})$ be the $\mathcal{M}^{\ast}$-marked connection defined in \ref{extformalmodel}. Such a connection exists at the cost of replacing $X$ by a small enough neighbourhood of $P$ in $X$. From proposition \ref{sietendbon}, we are left to show that $(\mathcal{N}^{\ast}, \iso^{\ast})$ extends into a $\mathcal{M}$-marked connection. Let $\Delta$ be a small enough disc in $D$ containing $P$ such that $\Irr_{D}^{\ast} \mathcal{M}$ and $\Irr_{D}^{\ast} \End\mathcal{M}$ are local systems on $\Delta$.  Put $\Delta^{\ast}:=\Delta\setminus \{P\}$. Let $\pi : Y\longrightarrow X$ be a resolution of turning points for $\mathcal{M}$ at $P$. Set $E:=\pi^{-1}(\Delta)$ and pick $Q\in \Delta^{\ast}$. Let 
$$
\xymatrix{
\Phi : H^{1}(\partial E, \St_{\pi^{+}\mathcal{M}}^{<E})            \ar[r] & H^{1}(\partial Q, \St_{\mathcal{M}}^{<\Delta})       
}
$$
be the restriction morphism of Stokes torsors. From lemma \ref{exte}, to prove that $(\mathcal{N}^{\ast}, \iso^{\ast})$ extends into a $\mathcal{M}$-marked connection, it is enough to prove that $(\mathcal{N}^{\ast}_Q, \iso^{\ast}_Q)$  lies in the image of $\Phi$. This is indeed the case by lemma \ref{smoothproper}, which finishes the proof of Theorem $1$.

\section{A boundedness theorem for turning points}

\subsection{Nearby slopes}
Let $X$ be a smooth complex algebraic variety and let $\mathcal{M}$ be an holonomic  $\mathcal{D}_X$-module. Let $f\in \mathcal{O}_X$ be a non constant function. Let $\psi_f$ be the nearby cycle functor associated to $f$ \cite{Kashpsi}\cite{Malpsi}\cite{MS}\cite{MM}. Following \cite{NearbySlope}, we recall that the \textit{nearby slopes of $\mathcal{M}$ associated to $f$} are the rational numbers $r\in \mathds{Q}_{\geq 0}$ such that there exists a germ $N$ of meromorphic connection at $0\in \mathds{A}^1_{\mathds{C}}$ with slope $r$ such that 
\begin{equation}\label{annulation0}
\psi_{f}(\mathcal{M}\otimes f^{+}N)\neq 0
\end{equation}
We denote by $\Sl_{f}^{\nb}(\mathcal{M})$ the set of nearby slopes of $\mathcal{M}$ associated to $f$. In dimension $1$, the nearby slopes of $\mathcal{M}$ associated to a local coordinate centred at a point $0$ are the usual slopes of $\mathcal{M}$  at $0$. See \cite[3.3.1]{NearbySlope} for a proof. In general, the set $\Sl_{f}^{\nb}(\mathcal{M})$  is finite \cite{DeligneLettreMalgrange}.
If $\mathcal{M}$ is a meromorphic connection, an explicit bound for $\Sl_{f}^{\nb}(\mathcal{M})$ is given in \cite{NearbySlope} in terms of a resolution of turning points of $\mathcal{M}$. This bound behaves poorly with respect to restriction. We will need a sharper bound in the case where $f$ is a smooth morphism. It will be provided by the following more general proposition.
\begin{proposition}\label{sharpbound}
Let $\mathcal{M}$ be a germ of meromorphic connection at $0\in \mathds{A}^n$ with poles along the  divisor $D$ given by $f:=x_1 \cdots x_d =0$. Let $r_i$ be the  highest generic slope of $\mathcal{M}$  along $x_i=0$. Put $r_D(\mathcal{M})=\Max \{r_1, \dots , r_d\}$. Then, 
$$
\Sl_{f}^{\nb}(\mathcal{M})\subset [0, r_D(\mathcal{M})]
$$
\end{proposition}
\begin{proof}
To prove proposition \ref{sharpbound}, take $r>r_D(\mathcal{M})$ and let $N$ be a  germ of meromorphic connection at $0\in \mathds{A}^1_{\mathds{C}}$ with slope $r$. We want to show the vanishing \eqref{annulation0} in a neigbourhood of $0$. By a standard Galois argument, one reduces to the case where $r$ and the $r_i, i=1, \dots , d$ are integers. Since $\psi_f$ is a formal invariant, we can further suppose that $N=t^\alpha \mathcal{E}^{1/t^r}$ where $\alpha \in \mathds{C}$. Let us accept for a moment that $\mathcal{M}$ is generated as a  $\mathcal{D}_X$-module by a coherent $\mathcal{O}_X$-submodule $F$ stable by 
$f^{r_D(\mathcal{M})}x_i\partial_{x_i}$, $i=1, \dots , d$ and such that $\mathcal{M}=F(\ast D)$. Let $(e_1, \dots ,e_N)$ be a generating family for $F$ in a neighbourhood of $0$. Then, the $ f^{\alpha}e^{1/f^r}e_i$, $i=1, \dots, N$ generate $\mathcal{M}\otimes f^{+}N$ as a $\mathcal{D}_X$-module. Let $\iota :  \mathds{C}^n\longrightarrow \mathds{C}^n\times \mathds{C}_t$ be the graph of $f$. Set $\delta:=\delta(t-f)$. Then, the $s_i= f^{\alpha}e^{1/f^r}e_i \delta$, $i=1, \dots, N$ generate $\iota_+ (\mathcal{M}\otimes f^{+}N)$. To show that the germ of $\psi_{f}(\mathcal{M}\otimes f^{+}N)$ at $0$ vanishes, we are thus left to prove that $s_i$ belongs to $V_{-1}(\mathcal{D}_{\mathds{C}^n\times \mathds{C}_t})\iota_+F$ for every $i=1, \dots, N$, where $V_{\bullet}(\mathcal{D}_{\mathds{C}^n\times \mathds{C}_t})$ is the Kashiwara-Malgrange filtration on $\mathcal{D}_{\mathds{C}^n\times \mathds{C}_t}$.
For $i=1, \dots, N$, we have 
$$
f^{r_D(\mathcal{M})}x_1\partial_{x_1}s_i = f^{r_D(\mathcal{M})}(\alpha -\frac{r}{f^r})s_i +\sum_{j=1}^{d} g_js_j -f^{r_D(\mathcal{M})+1}\partial_t s_i  ,   \hspace{1cm} g_j\in \mathcal{O}_X
$$
Hence,
$$
r s_i =  \alpha t^r s_i + t^{r-r_D(\mathcal{M}) }\sum_{j=1}^{d} g_js_j - f^r x_1\partial_{x_1}s_i - f^{r+1}\partial_t s_i
$$
Since $r>r_D(\mathcal{M})$, we have 
$$
t^{r-r_D(\mathcal{M}) }\sum_{j=1}^{d} g_js_j \in V_{-1}(\mathcal{D}_{\mathds{C}^n\times \mathds{C}_t})\iota_+F
$$
Note furthermore that
$$
f^r x_1\partial_{x_1} s_i =  x_1\partial_{x_1} f^r s_i - r f^r s_i = t^r( x_1\partial_{x_1} -r)s_i     \in V_{-1}(\mathcal{D}_{\mathds{C}^n\times \mathds{C}_t}) s_i
$$
and that 
$$
f^{r+1}\partial_t s_i =  \partial_t t^{r+1} s_i = (r+1)t^r s_i + t^r t\partial_t s_i \in V_{-1}(\mathcal{D}_{\mathds{C}^n\times \mathds{C}_t}) s_i
$$
Hence, $s_i\in V_{-1}(\mathcal{D}_{\mathds{C}^n\times \mathds{C}_t})\iota_+F$, which proves the sought-after vanishing. We are thus left to prove the lemma \ref{malgrange} below.
\end{proof}
\begin{lemme}\label{malgrange}
Let $\mathcal{M}$ be a germ of meromorphic connection at $0\in \mathds{A}^n_{\mathds{C}}$ with poles along the  divisor $D$ given by $f:=x_1, \cdots x_d =0$. Let $r_i$ be the  highest generic slope of $\mathcal{M}$  along $x_i=0$. Suppose that the $r_i$ are integers and put $r_D(\mathcal{M})=\Max \{r_1, \dots , r_d\}$. Then,  $\mathcal{M}$ is generated as a  $\mathcal{D}_X$-module by a coherent $\mathcal{O}_X$-submodule $F$ stable by 
$f^{r_D(\mathcal{M})} x_i\partial_{x_i}$ for every $i=0, \dots, d$ and such that $\mathcal{M}=F(\ast D)$. 
\end{lemme}
\begin{proof}
Let $E$ be a lattice in $\mathcal{M}$ as constructed by Malgrange in \cite{Reseaucan}. By construction, $\mathcal{M}=E(\ast D)$. Since holonomic $\mathcal{D}_X$-modules are noetherian, $\mathcal{M}= \mathcal{D}_X f^{-k} E$ for $k$ big enough. Let us show that $F=f^{-k} E$ fits our purpose. For $m\in E$, we have
$$
f^{r_D(\mathcal{M})}x_i\partial_{x_i}(f^{-k} m)=-k f^{r_D(\mathcal{M})}(f^{-k}m)+f^{-k}(f^{r_D(\mathcal{M})}x_i\partial_{x_i}m)
$$
Hence, it is enough to show that $E$ is stable by $f^{r_D(\mathcal{M})}x_i\partial_{x_i}$, $i=0, \dots, d$. Since $\mathcal{O}_{X^{\an},x}$ is  faithfully flat over $\mathcal{O}_{X,x}$ for every $x\in D$, we have $E= \mathcal{M}\cap E^{\an}$ in $\mathcal{M}^{\an}$. Hence, it is enough to show that $E^{\an}$ is stable by $f^{r_D(\mathcal{M})}x_i\partial_{x_i}$, $i=0, \dots, d$. Let $j : U\longrightarrow X^{\an}$ be the complement in $X$ of the union of the singular locus of $D$ with the turning locus of $\mathcal{M}$. By construction of the Deligne-Malgrange lattices, a section of $\mathcal{M}$ belongs to $E^{\an}$ if and only if its restriction to $U$ belongs to $E^{\an}_{|U}$. Hence, we can suppose that $D$ is smooth and that $\mathcal{M}$ has good formal structure along $D$. We can further suppose that $\mathcal{M}$ is unramified along $D$. Since $\mathcal{O}_{X^{\an},\hat{D}}$ is  faithfully flat over $\mathcal{O}_{X^{\an}|D}$, we can suppose that  $\mathcal{M}$ splits, that is 
$$
\mathcal{M}=\bigoplus_{a\in \mathcal{O}_{X^{\an}}(r_D(\mathcal{M})D)}\mathcal{E}^{a} \otimes \mathcal{R}_a
$$
where the $\mathcal{R}_a$ are regular meromorphic connections with poles along $D$. In that case, $E$ is by definition a direct sum of the form $\bigoplus E_a$ where $E_a$ is a  Deligne lattice \cite{Del} in $\mathcal{R}_a$. In that case, the sought-after stability is obvious. This  finishes the proof of lemma \ref{malgrange}.
\end{proof}

\begin{remarque}\label{lienladic}
The bound for nearby slopes proved in proposition \ref{sharpbound} was suggested by the $\ell$-adic picture \cite{HuTeyssier}.  In \textit{loc. it.} indeed, a similar bound was obtained for $\ell$-adic nearby slopes of smooth morphisms \cite{Teyladic}. In that setting, the main tools are Beilinson's and Saito's work on the singular support \cite{Bei} and the characteristic cycle \cite{SaitoInv} for $\ell$-adic sheaves, as well as semi-continuity properties \cite{HuYang}\cite{HuCurves} for various ramification invariants produced by Abbes and Saito's ramification theory \cite{AbbesSaitoImperfect}. From this perspective, proposition \ref{sharpbound} is a positive answer to a local variant for differential equations of a conjecture in \cite{Leal} on the  ramification of the étale cohomology groups for local systems on the generic fiber of a strictly semi-stable pair. See Conjecture 5.8 from \cite{HuTeyssier} for a precise statement.
\end{remarque}

\subsection{Boundedness of the turning locus in the case of smooth proper relative curves}
This subsection is devoted to the proof of Theorem \ref{boundedness}. Let $S$ be a smooth complex algebraic curve. Let $p: \mathcal{C}\longrightarrow S$ be a relative smooth proper curve of genus $g$. Let $\mathcal{M}$ be a meromorphic connection of rank $r$ on  $\mathcal{C}$ with poles along the fibre $\mathcal{C}_0$. Let $Z(\mathcal{M})$ be  the subset of points in 
$\mathcal{C}_0$ at which $\mathcal{M}$ does not have good formal structure (that is, the turning locus of $\mathcal{M}$). Let $\irr_{\mathcal{C}_0}\mathcal{M}$ be the generic irregularity of $\mathcal{M}$ along $\mathcal{C}_0$. Let $r_D(\mathcal{M})$ be the highest generic slope of $\mathcal{M}$ along $\mathcal{C}_0$. We put
$$
K:= (\Sol \mathcal{M})_{|\mathcal{C}_0}[1] \oplus  (\Sol \End\mathcal{M})_{|\mathcal{C}_0}[1]
$$
Then, $K$ is a complex of $\mathds{C}$-vector spaces on $\mathcal{C}_0$ with constructible cohomology. It is concentrated in degree $0$ and $1$. The generic rank of $K$ is 
\begin{align*}
r_{K}&=\irr_{\mathcal{C}_0}\mathcal{M}+\irr_{\mathcal{C}_0}\End\mathcal{M}\\
       & \leq r r_D(\mathcal{M} )+ r^2 r_D(\End\mathcal{M} )\\
       &\leq 2r^2 r_D(\mathcal{M} )
\end{align*}
where the last inequality comes from $r_D(\End\mathcal{M} )\leq r_D(\mathcal{M} )$. The Euler-Poincaré characteristic formula \cite[Th. 2.2.1.2]{Laumon} applied to $K$ gives
$$
\chi(\mathcal{C}_0, K)=(2-2g)r_{K}-\displaystyle{\sum_{x\in \Sing K  } }  (r_{K}-\dim \mathcal{H}^0K_x )+\dim \mathcal{H}^1K_x
$$
where $\Sing K$ denotes the singular locus of $K$, that is the subset of points in $\mathcal{C}_0$ in the neighbourhood of which $K$ is not a local system concentrated in degree $0$. From Mebkhout perversity theorem \cite{Mehbgro}, the complex $K$ is perverse. In particular, $\mathcal{H}^0K$ does not have sections with punctual support. Thus, 
$$
r_{K}-\dim \mathcal{H}^0K_x\geq 0
$$
for every $x\in \Sing K$. From perversity again \cite[13.1.6]{phdteyssier}, the local Euler-Poincaré characteristic of $K$ at $x\in \Sing K$ differs from its generic value $r_{K}$. Hence, for $x\in \Sing K$, the quantity
$$
(r_{K}-\dim \mathcal{H}^0K_x )+\dim \mathcal{H}^1K_x
$$
is positive and non zero. It is thus strictly positive. Hence, we have a bound
$$
 |\Sing K| \leq (2-2g)r_{K} - \chi(\mathcal{C}_0, K)
$$
From Theorem \ref{mainth}, the singular points of $K$ are exactly the points  in $\mathcal{C}_0$ at which $\mathcal{M}$ does not have good formal structure. Hence
$$
| Z(\mathcal{M})| \leq 2r_{K} + |\chi(\mathcal{C}_0, K)|
$$
We are now left to bound $\chi(\mathcal{C}_0, K)$. Since the irregularity complex is compatible with proper push-forward \cite[3.6-6]{Mehbsmf}, we have
\begin{align*}
|\chi(\mathcal{C}_0, (\Sol \mathcal{M})_{|\mathcal{C}_0}) |&=|\chi(0, Rp_{\ast }(\Sol \mathcal{M})_{|\mathcal{C}_0})|\\
&=|\chi(0, (\Sol p_+ \mathcal{M})_{|0})|\\
                                  &=|\sum_i (-1)^{i}\irr_0 \mathcal{H}^i p_+\mathcal{M}|\\
                                   &\leq \sum_i  \irr_0 \mathcal{H}^i p_+\mathcal{M}\\
                                  &\leq \sum_i\rk\mathcal{H}^i p_+\mathcal{M} 
                                 \times \Max \Sl^{\nb}_t(\mathcal{H}^ip_+\mathcal{M})\\
                                 &\leq \sum_i\rk\mathcal{H}^i p_+\mathcal{M} 
                                 \times \Max \Sl^{\nb}_t(p_+\mathcal{M})
\end{align*}
Since nearby slopes are compatible with proper push-forward \cite[Th. 3 $(ii)$]{NearbySlope}, we have $\Sl^{\nb}_t(p_+\mathcal{M}) \subset \Sl^{\nb}_p(\mathcal{M})$. Since $p$ is smooth, proposition \ref{sharpbound} yields
$$
|\chi(\mathcal{C}_0, (\Sol \mathcal{M})_{|\mathcal{C}_0}) |\leq 
r_D(\mathcal{M} )\sum_i\rk\mathcal{H}^i p_+\mathcal{M}                                 
$$
For a generic point $s\in S$, we have furthermore
\begin{align*}
\sum_i\rk\mathcal{H}^i p_+\mathcal{M} &=\sum_i\dim (\Sol \mathcal{H}^i p_+\mathcal{M})_s \\
                                & = \sum_i\dim (\mathcal{H}^i\Sol  p_+\mathcal{M})_s\\
                                & = \sum_i\dim (R^ip_\ast\Sol  \mathcal{M})_s \\
                                &=\sum_i\dim H^{i}(\mathcal{C}_s,\mathcal{L} )
\end{align*}
where $\mathcal{L}$ denotes the local system of solutions of  $\mathcal{M}_{|\mathcal{C}_s}$. Then $H^{i}(\mathcal{C}_s,\mathcal{L} )=0$ for every $i\neq 0, 1, 2$ and we have
$$
\dim H^{0}(\mathcal{C}_s,\mathcal{L} )\leq \rk \mathcal{L}=   \rk \mathcal{M}_{|\mathcal{C}_s}= r
$$
From Poincaré-Verdier duality, we have
$$
\dim H^{2}(\mathcal{C}_s,\mathcal{L} )= \dim H^{0}(\mathcal{C}_s,\mathcal{L}^{\ast} )\leq \rk \mathcal{L}^{\ast}=  \rk \mathcal{M}_{|\mathcal{C}_s}= r
$$
Finally, 
\begin{align*}
\dim H^{1}(\mathcal{C}_s,\mathcal{L} ) & = -\chi(\mathcal{C}_s,\mathcal{L} ) +\dim H^{0}(\mathcal{C}_s,\mathcal{L} )+\dim H^{2}(\mathcal{C}_s,\mathcal{L} )\\
                               &= -\chi(\mathcal{C}_s, \mathds{C})\rk \mathcal{L}+\dim H^{0}(\mathcal{C}_s,\mathcal{L} )+\dim H^{2}(\mathcal{C}_s,\mathcal{L} ) \\
                               &\leq 2r(g+1)                               
\end{align*}
Putting everything together yields
$$
| Z(\mathcal{M})| \leq 8r^2(g+1)r_D(\mathcal{M} )
$$
This finishes the proof of Theorem \ref{boundedness}.

\bibliographystyle{amsalpha}
\bibliography{Saito-Sato-Seminar}

\providecommand{\bysame}{\leavevmode\hbox to3em{\hrulefill}\thinspace}
\providecommand{\MR}{\relax\ifhmode\unskip\space\fi MR }
\providecommand{\MRhref}[2]{%
  \href{http://www.ams.org/mathscinet-getitem?mr=#1}{#2}
}
\providecommand{\href}[2]{#2}
\begin{thebibliography}{{Moc}11b}

\bibitem[{And}07]{Andre}
Y.~{André}, \emph{Structure des connexions méromorphes formelles de plusieurs
  variables et semi-continuité de l'irrégularité}, Invent. math.
  \textbf{170} (2007).

\bibitem[AS02]{AbbesSaitoImperfect}
A.~{Abbes} and T.~{Saito}, \emph{{Ramification of local fields with imperfect
  residue fields}}, {{Amer. J. Math.}} \textbf{124} (2002).

\bibitem[{Bei}16]{Bei}
A.~{Beilinson}, \emph{{Constructible sheaves are holonomic}}, {{Sel. Math. New.
  Ser.}} \textbf{22} (2016).

\bibitem[{Bor}91]{BorelAlgGp}
A.~{Borel}, \emph{{Linear Algebraic Groups, {Second Enlarged Edition}}},
  {Graduate Texts in Mathematics}, 1991.

\bibitem[BV89]{BV}
D.G. {Babbitt} and V.S. {Varadarajan}, \emph{{Local Moduli For Meromorphic
  Differential Equations}}, Astérisque, vol. 169-170, 1989.

\bibitem[CDG17]{CDG}
G.~{Cotti}, B.~{Dubrovin}, and D.~{Guzzetti}, \emph{Isomonodromy deformations
  at an irregular singularity with coalescing eigenvalues}, Preprint,
  \url{https://arxiv.org/pdf/1706.04808.pdf}, 2017.

\bibitem[{Del}70]{Del}
P.~{Deligne}, \emph{Equations différentielles à points singuliers
  réguliers}, Lecture {N}otes in {M}athematics, vol. 163, Springer-{V}erlag,
  1970.

\bibitem[{Del}07]{DeligneLettreMalgrange}
\bysame, \emph{Lettre à {M}algrange. 20 décembre 1983}, Singularités
  irrégulières (Société {M}athématique~de {F}rance, ed.), Documents
  {M}athématiques, vol.~5, 2007.

\bibitem[EK12]{HenKerz}
H.~{Esnault} and M.~{Kerz}, \emph{{A finiteness theorem for {G}alois
  representations of function fields over finite fields (after {Deligne})}},
  {Acta Mathematica Vietnamica} \textbf{4} (2012).

\bibitem[{Fre}57]{JFrenkel}
J.~{Frenkel}, \emph{Cohomologie non abélienne et espaces fibrés}, {{Bulletin
  de la {S.M.F}}} \textbf{85} (1957).

\bibitem[HE17]{HuYang}
H.~{Hu} and E.{Yang}, \emph{{Semi-continuity for total dimension divisors of
  étale sheaves}}, {{Internat. J. Math.}} \textbf{28} (2017).

\bibitem[{Hie}09]{HienInv}
M.~{Hien}, \emph{Periods for flat algebraic connections}, Inventiones
  {M}athematicae \textbf{178} (2009).

\bibitem[HT18]{HuTeyssier}
H.~{Hu} and J.-B. {Teyssier}, \emph{Characteristic cycle and wild ramification
  for nearby cycles of étale sheaves}, Preprint, 2018.

\bibitem[{Hu}17]{HuCurves}
H.~{Hu}, \emph{{Logarithmic ramifications of \'etale sheaves by restricting to
  curves}}, {{IMRN}} \textbf{rnx228} (2017).

\bibitem[{Kas}73]{Kajap}
M.~{Kashiwara}, \emph{Index theorem for a maximally overdetermined system of
  linear differential equations}, Proc. {J}ap {A}cad., vol. 49-10, 1973.

\bibitem[{Kas}75]{Ka2}
\bysame, \emph{On the maximally overdetermined systems of linear differential
  equations {I}}, Publ. {RIMS} \textbf{10} (1975).

\bibitem[{Kas}83]{Kashpsi}
\bysame, \emph{Vanishing cycle sheaves and holonomic systems of differential
  equations}, Algebraic {G}eometry (Springer, ed.), Lecture {N}otes in
  {M}athematics, vol. 1016, 1983.

\bibitem[{Kas}95]{TheseKashiwara}
\bysame, \emph{{Algebraic study of systems of partial differential equations.
  Thesis, Tokyo University, 1970 (Translated by Andrea D'Agnolo and Pierre
  Schneiders, with a foreword by Pierre Schapira)}}, Mémoire de la {SMF}
  \textbf{63} (1995).

\bibitem[{Ked}10]{Kedlaya1}
K.~{Kedlaya}, \emph{Good formal structures for flat meromorphic connections
  {I}: {S}urfaces}, Duke Math.J. \textbf{154} (2010).

\bibitem[{Ked}11]{Kedlaya2}
\bysame, \emph{Good formal structures for flat meromorpohic connexions {II}:
  excellent schemes}, J. Amer. Math. Soc. \textbf{24} (2011).

\bibitem[{Ked}18]{Kedlaya3}
\bysame, \emph{{Good formal structures for flat meromorphic connections, III:
  Irregularity and turning loci}}, \url{https://arxiv.org/pdf/1308.5259.pdf},
  2018.

\bibitem[KS90]{KS}
M.~{Kashiwara} and P.~{Schapira}, \emph{Sheaves on manifolds},
  Springer-{V}erlag, 1990.

\bibitem[{Lau}87]{Laumon}
G.~{Laumon}, \emph{Transformation de {F}ourier, constantes d'{\'e}quations
  fonctionnelles et conjecture de {W}eil}, Publications Math{\'e}matiques de
  l'IH{\'E}S \textbf{65} (1987).

\bibitem[{Lea}16]{Leal}
I.~{Leal}, \emph{{On the ramification of \'etale cohomology groups}}, {{J.
  reine angew. Math.}} ({2016}).

\bibitem[{Mal}81]{Bouka}
B.~{Malgrange}, \emph{Rapport sur les théorèmes d'indice de {B}outet de
  {M}onvel et {K}ashiwara}, Analyse et topologie sur les espaces singuliers
  {II-III}, Astérisque, vol. 101-102, Soc. Math. France, 1981.

\bibitem[{Mal}83a]{Mal2MathPhy}
\bysame, \emph{Mathématique et {P}hysique}, vol.~37, ch.~La classification des
  connections irrégulières à une variable, Birkhäuser, 1983.

\bibitem[{Mal}83b]{Malpsi}
\bysame, \emph{Polynômes de {B}ernstein-{S}ato et cohomologie évanescente},
  Astérisque, vol. 101-102, 1983.

\bibitem[{Mal}96]{Reseaucan}
\bysame, \emph{Connexions méromorphes 2: Le réseau canonique}, Inv. {M}ath.
  \textbf{124} (1996).

\bibitem[{Meb}89]{MS}
Z.~{Mebkhout}, \emph{Le formalisme des six op{\'e}rations de {G}rothendieck
  pour les $\mathcal{D}$-modules coh{\'e}rents}, vol.~35, Hermann, 1989.

\bibitem[{Meb}90]{Mehbgro}
\bysame, \emph{Le théorème de positivité de l’irrégularité pour les
  $\mathcal{D}_{X}$-modules}, The {G}rothendieck {F}estschrift {III}, vol.~88,
  Birkhäuser, 1990.

\bibitem[{Meb}04]{Mehbsmf}
\bysame, \emph{Le théorème de positivité, le théorème de comparaison et le
  théorème d'existence de {R}iemann}, Éléments de la théorie des systèmes
  différentiels géométriques, Cours du {C.I.M.P.A.}, Séminaires et
  Congrès, vol.~8, SMF, 2004.

\bibitem[MM04]{MM}
P.~{Maisonobe} and Z.~{Mebkhout}, \emph{Le th{\'e}or{\`e}me de comparaison pour
  les cycles {\'e}vanescents}, Éléments de la théorie des systèmes
  différentiels géométriques, Cours du {C.I.M.P.A.}, Séminaires et
  Congrès, vol.~8, SMF, 2004.

\bibitem[{Moc}09]{Mochizuki2}
T.~{Mochizuki}, \emph{{Good formal structure for meromorphic flat connections
  on smooth projective surfaces.}}, {Algebraic analysis and around in honor of
  Professor Masaki Kashiwara's 60th birthday}, Tokyo: Mathematical Society of
  Japan, 2009.

\bibitem[{Moc}11a]{MochStokes}
\bysame, \emph{The {S}tokes structure of a good meromorphic flat bundle},
  Journal of the Institute of Mathematics of Jussieu \textbf{10} (2011).

\bibitem[{Moc}11b]{Mochizuki1}
\bysame, \emph{Wild {H}armonic {B}undles and {W}ild {P}ure {T}wistor
  $\mathcal{D}$-modules}, Ast{\'e}risque, vol. 340, SMF, 2011.

\bibitem[{Sab}00]{Sabbahdim}
C.~{Sabbah}, \emph{Equations diff{\'e}rentielles {\`a} points singuliers
  irr{\'e}guliers et ph{\'e}nom{\`e}ne de {S}tokes en dimension 2},
  Ast{\'e}risque, vol. 263, {SMF}, 2000.

\bibitem[{Sab}02]{Frob}
\bysame, \emph{D{\'e}formations isomonodromiques et vari{\'e}t{\'e}s de
  {F}robenius}, Savoirs actuels, vol. 2060, {CNRS} Edition, 2002.

\bibitem[{Sab}17a]{SabRemar}
\bysame, \emph{{A remark on the {Irregularity Complex}}}, {Journal of
  Singularities} \textbf{16} (2017).

\bibitem[{Sab}17b]{SabbahsurCDG}
\bysame, \emph{Deformations and degenerations of irregular singularities},
  Preprint, \url{https://arxiv.org/abs/1711.08514}, 2017.

\bibitem[{Sai}17]{SaitoInv}
T.~{Saito}, \emph{{The characteristic cycle and the singular support of a
  constructible sheaf}}, {{Invent. Math.}} \textbf{207(2)} (2017).

\bibitem[{Sim}94]{SimpsonII}
C.~{Simpson}, \emph{{Moduli of representations of the fundamental group of a
  smooth projective variety II}}, {{Publ. Math. de l'IHES}} \textbf{80} (1994).

\bibitem[SPD]{SPDescent}
\emph{{Stack Project}}, ch.~Descent.

\bibitem[Sv00]{SVDP}
M.T {Singer} and M.~{van der Put}, \emph{Galois {T}heory of {L}inear
  {D}ifferential {E}quations}, Grundlehren der mathematischen Wissenschaften,
  vol. 328, Springer, 2000.

\bibitem[{Tey}13]{phdteyssier}
J.-B. {Teyssier}, \emph{Autour de l'irrégularité des connexions
  méromorphes}, Ph.D. thesis, Ecole Polytechnique, 2013.

\bibitem[{Tey}14]{LetterToSabbah}
\bysame, \emph{Mail to {C. Sabbah}}, May 2014.

\bibitem[{Tey}15]{Teyladic}
\bysame, \emph{Nearby slopes and boundedness for $\ell$-adic sheaves in
  positive characteristic. {P}reprint}, 2015.

\bibitem[{Tey}16]{NearbySlope}
\bysame, \emph{A boundedness theorem for nearby slopes of holonomic
  $\mathcal{D}$-modules}, {{Compositio Mathematica}} \textbf{152} (2016).

\bibitem[{Tey}17]{TeySke}
\bysame, \emph{Skeletons and moduli of {S}tokes torsors}, {{To appear at
  Annales Scientifiques de l'Ecole Normale Supérieure}} (2017).

\bibitem[{Tey}19]{Rare}
\bysame, \emph{{Higher dimensional Stokes structures are rare}}, Preprint,
  2019.

\end{thebibliography}

\end{document}